\documentclass{amsart}

\usepackage{amsmath, amsthm, amssymb, xypic, setspace}
\usepackage{hyperref}
\usepackage[usenames, dvipsnames]{color}

\hypersetup{
    colorlinks = true,
    linkcolor = RoyalBlue,
    citecolor = Cerulean
}

\numberwithin{equation}{section}

\theoremstyle{definition}
\newtheorem{definition}{Definition}[section]

\newtheorem{example}[definition]{Example}
\newtheorem{remark}[definition]{Remark}

\theoremstyle{theorem}
\newtheorem{proposition}[definition]{Proposition}
\newtheorem{theorem}[definition]{Theorem}

\newtheorem{lemma}[definition]{Lemma}

\newcommand{\norm}[1]{\left\lVert#1\right\rVert}
\newcommand{\ddb}{\partial\bar{\partial}}

\title{On Yau's theorem for effective orbifolds}
\author{Mitchell Faulk}
\address{Department of Mathematics, Columbia University, New York, NY}
\email{faulk@math.columbia.edu}

\begin{document}

\maketitle

\begin{abstract}
In 1978, Yau \cite{yau78} confirmed a conjecture due to Calabi \cite{calabi54} stating the existence of K\"ahler metrics with prescribed Ricci forms on compact K\"ahler manifolds. A version of this statement for effective orbifolds can be found in the literature \cite{joyce00, bg08, dk01}. In this expository article, we provide details for a proof of this orbifold version of the statement by adapting Yau's original continuity method to the setting of effective orbifolds in order to solve a Monge-Amp\`ere equation. We then outline how to obtain K\"ahler-Einstein metrics on orbifolds with $c_1(\mathcal{X}) < 0$ by solving a slightly different Monge-Amp\`ere equation. We conclude by listing some explicit examples of Calabi-Yau orbifolds, which consequently admit Ricci flat metrics by Yau's theorem for effective orbifolds. 
\end{abstract}

\medskip


\tableofcontents

\section{Introduction}

The notion of effective orbifold (or equivalently $V$-manifold) is a generalization of the notion of manifold whereby the local charts are required to be homeomorphic to a quotient $U/G$ of an open subset $U$ of $\mathbb{R}^n$ by a finite group $G$. If the defining data of an orbifold is holomorphic, then the orbifold is called complex. 

Many results from complex geometry, such as the Hodge decomposition theorem for K\"ahler manifolds \cite{baily56} and Kodaira's embedding theorem \cite{baily57}, have been shown to extend from the setting of manifolds to the setting of effective orbifolds. In addition, it is stated in the literature \cite[Theorem 6.56]{joyce00} \cite[Theorem 4.4.31]{bg08} \cite[Technical Setting 6.2]{dk01} that one such result, namely, Yau's solution \cite{yau78} to Calabi's conjecture \cite{calabi54}, extends as well. Our purpose is to provide a self-contained proof of this result, which we state precisely below. 
 
\begin{theorem}\label{thm:main}
Let $(\mathcal{X}, \omega)$ be a compact K\"ahler effective orbifold of complex dimension $n$, and let $F$ be a smooth function on $\mathcal{X}$ satisfying 
\[
\int_\mathcal{X} e^F \omega^n = \int_{\mathcal{X}} \omega^n.
\]
Then there is a smooth function $\varphi$ on $\mathcal{X}$, unique up to the addition of a constant, satisfying 
\begin{align}\label{eqn:MA}
\begin{cases}
(\omega + \sqrt{-1}\ddb \varphi)^n = e^F \omega^n  \\
\omega + \sqrt{-1} \ddb \varphi \; \textnormal{is a positive form}
\end{cases}. 
\end{align}
\end{theorem} 

It is easily shown that Theorem \ref{thm:main} implies---or, more precisely, is equivalent to---the following, which is the analogous extension of the original Calabi conjecture \cite{calabi54} to the setting of orbifolds.  

\begin{theorem}\label{thm:calabi}
Let $(\mathcal{X}, \omega)$ be a compact K\"ahler effective orbifold, and let $R$ be a $(1,1)$-form representing the cohomology class $2\pi c_1(\mathcal{X}) \in H^2(\mathcal{X}, \mathbb{R})$. Then there is a unique K\"ahler form $\omega'$ on $\mathcal{X}$ such that
\begin{enumerate}
\item[(i)] $\omega'$ and $\omega$ represent the same cohomology class and
\item[(ii)] the Ricci form of $\omega'$ is $R$. 
\end{enumerate}
\end{theorem}

In particular, if $c_1(\mathcal{X}) = 0$ as a cohomology class in $H^2(\mathcal{X}, \mathbb{R})$, then there is a unique Ricci flat K\"ahler form on $\mathcal{X}$ (c.f. \cite[Theorem 1.3]{campana04}). 

If one instead solves a slightly modified Monge-Amp\`ere equation from the one above, namely 
\begin{align}\label{eqn:MA2}
\begin{cases}
(\omega + \sqrt{-1} \ddb \varphi)^n = e^{F + \varphi}\omega^n \\
\omega + \sqrt{-1} \ddb \varphi \; \text{is a positive form}
\end{cases},
\end{align}
on a compact K\"ahler orbifold whose first Chern class $c_1(\mathcal{X})$ is negative with K\"ahler form $\omega$ representing  $-2\pi c_1(\mathcal{X})$, then one obtains a K\"ahler-Einstein metric $\omega_\varphi = \omega + \sqrt{-1} \ddb \varphi$ satisfying $\text{Ric}(\omega_\varphi) = - \omega_\varphi$. Concisely, we have the following additional result concerning the existence of K\"ahler-Einstein metrics, due to Yau \cite{yau78} and Aubin \cite{a78} in the setting of manifolds. 

\begin{theorem}\label{thm:KE}
If $\mathcal{X}$ is a compact K\"ahler effective orbifold satisfying $c_1(\mathcal{X}) < 0$, then there is a K\"ahler metric $\omega \in -2\pi c_1(\mathcal{X})$ on $\mathcal{X}$ satisfying 
\[
\textnormal{Ric}(\omega) = -\omega. 
\]
\end{theorem}

Most of the content of this paper is far from original; instead well-known tools from the nonsingular setting are adapted to the setting of effective orbifolds. More precisely, our proof follows closely a streamlined approach to the continuity method of proof in the nonsingular setting, more details of which can be found in \cite{szekelyhidi14}, which compiles expositions due to many other sources \cite{siu12, tian12, blocki12, yau78}. We simply gather---and establish when necessary---appropriate orbifold-version ingredients of this approach, and then we outline how these ingredients can be used to construct a proof. The main ingredients include Kodaira's $\ddb$-lemma on orbifolds \cite{baily56}, properties of elliptic operators on orbifolds (see Section \ref{sec:elliptic}), a Green's function of the complex Laplacian, and inequalities of Poincar\'e and Sobolev type. We remark that other methods of proof (such as variational methods) exist in the nonsingular setting, and these approaches should have analogous extensions to the setting of orbifolds, provided the necessary ingredients can be extended to this setting as well. 

In Section \ref{sec:prelim}, we recall some basic facts of effective orbifolds. In particular, we review notions of tensors and differential forms on these objects, and we demonstrate a version of Stokes' theorem (c.f. \cite{baily56}). We also review some facts about K\"ahler orbifolds. 

Elliptic operators on orbifolds are studied in Section \ref{sec:elliptic}. We obtain a priori estimates for such operators, recall a construction of a Green's function of the Laplacian following \cite{chiang90}, and obtain inequalities of Poincar\'e and Sobolev type. 

Section \ref{sec:main} begins the proof of the main theorem. The continuity method is outlined, and the desired openness properties are established using the mapping properties of elliptic operators established in previous sections. 

We reserve Section \ref{sec:est} for the a priori estimates necessary for the closedness property of the continuity method. We outline how the estimates demonstrated in previous sections (e.g. Poincar\'e inequality and Sobolev inequality) can be used to establish $C^0$-, $C^2$-, and $C^3$-estimates on solutions to the desired equation. 

In Section \ref{sec:KE} we outline briefly how to obtain K\"ahler-Einstein metrics on orbifolds with negative first Chern class. 

Section \ref{sec:rf} gives examples of Ricci-flat metrics on orbifolds with $c_1(\mathcal{X}) = 0$ as a cohomology class in $H^2(\mathcal{X}, \mathbb{R})$, which are guaranteed by Theorem \ref{thm:calabi}.

\section{Orbifold preliminaries}\label{sec:prelim}

The goal of this section is to recall the notion of a classical effective orbifold and to review some differential geometric tools associated to these objects.

We follow closely the notation and terminology of \cite{alr07} for classical effective orbifolds. Let $X$ be a topological space. 
\begin{itemize}
\item An {orbifold chart} for $X$ consists of a triple $(\widetilde{U}, G, \pi)$ where $\widetilde{U}$ is an open connected subset of $\mathbb{R}^n$, $G$ is a finite group of smooth automorphisms of $\widetilde{U}$, and $\pi : \widetilde{U} \to X$ is a continuous map which is invariant under the action of $G$ and which induces a homeomorphism of $\widetilde{U}/G$ onto an open subset $U \subset X$, called the support of the chart $(\widetilde{U}, G, \pi)$.  
\item An {embedding} $\lambda : (\widetilde{U}, G, \pi) \hookrightarrow (\widetilde{V}, H, \rho)$ of one chart into another consists of a smooth embedding $\lambda : \widetilde{U} \to \widetilde{V}$ such that $\rho \circ \lambda = \pi$. 
\item An {orbifold atlas} is a family $\mathcal{U}$ of orbifold charts whose supports cover $X$ and which are compatible in the sense that if $x$ is a point in the intersection $U \cap V$ of the supports of two charts $ (\widetilde{U}, G, \pi)$ and $(\widetilde{V}, H, \rho)$, then there is an open neighborhood $W \subset U \cap V$ of $x$ and a chart $(\widetilde{W}, K, \sigma)$ for $W$ such that there are embeddings $(\widetilde{W}, K, \sigma) \hookrightarrow (\widetilde{U}, G, \pi)$ and $(\widetilde{W}, K, \sigma) \hookrightarrow (\widetilde{V}, H, \rho)$.
\item An atlas $\mathcal{U}$ is said to {refine} an atlas $\mathcal{V}$ if for each chart in $\mathcal{U}$ there is an embedding into some chart of $\mathcal{V}$. Two atlases are called {equivalent} if they have a common refinement. 
\end{itemize}

By an {effective orbifold $\mathcal{X}$} of dimension $n$ we mean a paracompact Hausdorff space $X$ equipped with an equivalence class $[\mathcal{U}]$ of $n$-dimensional orbifold atlases. We call $X$ the underlying space of the effective orbifold $\mathcal{X}$. Because we only deal with effective orbifolds in what follows, we use the term orbifold to mean effective orbifold.  From this point forward, because $X$ is paracompact, we will assume that $X$ is covered by a locally finite collection of supports $U_\alpha$ of an atlas of charts $(\widetilde{U}_\alpha, G_\alpha, \pi_\alpha)$.  

We say that an orbifold $\mathcal{X}$ is orientable if all of the smooth automorphisms and embeddings of the charts in an atlas are orientation-preserving. From this point forward, we will assume that $\mathcal{X}$ is orientable and equipped with an orientation.

By a tensor of type $(r,s)$ on $\mathcal{X}$ we mean a collection of tensors  $T_\alpha$ of type $(r,s)$ on the charts $(\widetilde{U}_\alpha, G_\alpha, \pi_\alpha)$ with the following properties.  
\begin{itemize}
\item Each $T_\alpha$ is invariant with respect to $G_\alpha$ in the sense that if $g \in G_\alpha$, then $g^*T_\alpha = T_\alpha$.
\item The collection $T_\alpha$ is compatible with the transition maps in the sense that if $\lambda_{\alpha\beta} : (\widetilde{U}_\alpha, G_\alpha, \pi_\alpha) \hookrightarrow (\widetilde{U}_\beta, G_\beta, \pi_\beta)$ is an embedding from the atlas, then $\lambda_{\alpha\beta}^* T_\beta = T_\alpha$.
\end{itemize}

The notion of $k$-form is defined similarly to that of a tensor. The naturality of the de Rham differential ensures that there is a de Rham operator $d$ taking $k$-forms to $(k+1)$-forms. 

The support of a tensor $T$ is the subset of the underlying space $X$ given by  
\[
\text{supp}(T) := \bigcup_\alpha \pi_\alpha (\text{supp}(T_\alpha)).
\]
Given a locally finite collection $V_\beta$ of open sets covering $X$, by a partition of unity subordinate to $V_\beta$, we mean a collection $\varphi_\beta$ of smooth (in the orbifold sense) functions on $\mathcal{X}$ satisfying 
\begin{itemize}
\item $\sum_\beta \varphi_\beta = 1$ and 
\item $\text{supp}(\varphi_\beta) \subset V_\beta$. 
\end{itemize}
(Note that each $\varphi_\beta$ itself is given as a collection of smooth functions $\varphi_{\beta\alpha}$, one on each $\widetilde{U}_\alpha$.) One can show that such partitions of unity exist for any locally finite collection $V_\beta$ (see \cite{baily56}).

The integral of a compactly supported $n$-form $\omega$ can be defined in the following manner. If $\omega$ is an $n$-form whose support is contained within the support $U_\alpha$ of a chart $(\widetilde{U}_\alpha, G_\alpha, \pi_\alpha)$, then define 
\[
\int_{\mathcal{X}} \omega = \frac{1}{|G_\alpha|} \int_{\widetilde{U}_\alpha} \omega_\alpha
\] 
where $\omega_\alpha$ is a corresponding $n$-form on $\widetilde{U}_\alpha$. More generally, for an arbitrary $n$-form $\omega$ on $\mathcal{X}$, choose a partition of unity $\varphi_\alpha$ subordinate to the supports $U_\alpha$ and define 
\[
\int_\mathcal{X} \omega = \sum_\alpha \int_\mathcal{X} \varphi_\alpha \omega.
\]
One can show that this definition does not depend on the choice of atlas in the equivalence class nor on the choice of partition of unity.

Stokes' theorem extends to compact orbifolds (without boundary) in a natural way. Indeed let $\omega$ be any $(n-1)$-form on a compact orbifold $\mathcal{X}$. Choose a partition of unity $\varphi_\alpha$ subordinate to the supports $U_\alpha$ of the orbifold charts. If we let $\tilde{\omega}_\alpha$ denote a $G_\alpha$-invariant $n$-form on $\widetilde{U}_\alpha$ representing $\varphi_\alpha \omega$, then we find by definition that  
\[
\int_\mathcal{X} d(\varphi_\alpha \omega) = \frac{1}{|G_\alpha|} \int_{\widetilde{U}_\alpha} d(\tilde{\omega}_\alpha) = 0
\]
where the latter integral vanishes by the ordinary Stokes' theorem, as $\tilde{\omega}_\alpha$ has support which is compact and contained within $\widetilde{U}_\alpha$. It now follows that 
\begin{align*}
\int_\mathcal{X} d\omega &= \int_\mathcal{X} d\left(\sum_\alpha \varphi_\alpha \omega \right) \\
&= \sum_\alpha \int_\mathcal{X} d(\varphi_\alpha \omega) \\
&= 0,
\end{align*}
where the interchanging of the sum and the integral sign is valid because, for example, we may suppose that the number of charts is finite as $\mathcal{X}$ is compact. We summarize below. 

\begin{lemma}[Stokes' theorem]
Let $\omega$ be an $(n-1)$-form on a compact orbifold $\mathcal{X}$. Then 
\[
\int_{\mathcal{X}} d\omega = 0.
\]
\end{lemma}

By a connection on $\mathcal{X}$ we mean an $\mathbb{R}$-linear mapping $\nabla$ taking $(1,0)$-tensors (vector fields) to $(1,1)$-tensors satisfying Leibniz rule 
\[
\nabla (fV) = df \otimes V + f \cdot \nabla V
\] 
for any smooth function $f$ and any vector field $V$ on $\mathcal{X}$.  For two vector fields $V,W$, we let $\nabla_V W$ denote the contraction of $\nabla W$ with $V$.  A connection is called symmetric if for any pair of vector fields $V,W$ on $\mathcal{X}$, we have $\nabla_VW - \nabla_WV = [V,W]$, where $[V,W]$ is the Lie bracket defined as the vector field satisfying $[V,W](f) = V(W(f)) - W(V(f))$ for any smooth function $f$ on $\mathcal{X}$. 
 
Any connection can be extended to a mapping from $(r,s)$-tensors to $(r,s+1)$ tensors by demanding that the extension be compatible with contraction and with Leibniz rule. In particular, for a Riemannian metric $g$ on $\mathcal{X}$, one finds that 
\[
d(g(V,W)) = (\nabla g)(V,W) + g(\nabla V, W) + g(V, \nabla W)
\]
for vector fields $V, W$.  A connection $\nabla$ is said to be compatible with the metric $g$ if $\nabla g = 0$. For a given metric $g$, there is a unique symmetric connection compatible with it, which we call the Levi-Civita connection and which we will denote by $\nabla$.

A Riemannian metric $g$ provides an identification of the space of $(1,0)$-tensors with the space of $(0,1)$-tensors. Thus for any two 1-forms $\eta$ and $\zeta$, the metric $g$ determines a smooth function $g(\eta, \zeta)$. More generally, for any pair of tensors (resp. forms) $S,T$ of the same type (resp. degree), we let $g(S,T)$ denote the smooth function so determined. In particular, we write $|S|_g^2$ for the smooth function 
\[
|S|_g^2 = g(S,S).
\]

A Riemannian metric $g$ also determines a unique $n$-form $dV$ called the volume form compatible with the orientation and the metric. In local coordinates, the volume form admits the expression 
\[
dV_\alpha = \sqrt{\det(g_\alpha)} \: dx_\alpha^1 \wedge \cdots \wedge dx_\alpha^n. 
\] 

An orbifold $\mathcal{X}$ is called complex if the charts can be taken to be subsets of $\mathbb{C}^n$ and the other defining data can be taken to be holomorphic. For such an orbifold, we have a well-defined almost complex structure $J$, that is, a mapping of vector fields to vector fields satisfying $J^2 = - \text{id}$, which can be described locally in the usual way. The complexification of the space of vector fields decomposes into the eigenspaces for $J$ corresponding to $\pm \sqrt{-1}$. Dually the complexification of the space of $1$-forms decomposes and a $1$-form corresponding to the eigenvalue $\sqrt{-1}$ (resp. $-\sqrt{-1}$) is called a $(1,0)$-form (resp. $(0,1)$-form). Taking higher exterior powers one obtains the notion of a $(p,q)$-form on $\mathcal{X}$. The extension of the de Rham operator to the complexified spaces decomposes as $d = \partial + \bar{\partial}$ where $\partial$ is an operator taking $(p, q)$-forms to $(p-1,q)$ forms and $\bar{\partial}$ an operator taking $(p,q)$-forms to $(p,q-1)$-forms. The relation $d^2 = 0$ implies that $\partial^2 = \bar{\partial}^2 = 0$. The following lemma is routine and follows from Stokes' theorem. 

\begin{lemma}[Integration by parts]
Let $T$ be a $d$-closed $(n-1, n-1)$ form and let $\varphi$ be a smooth function on a compact complex orbifold $\mathcal{X}$. Then 
\[
\int_\mathcal{X} \varphi \cdot  \ddb \varphi \wedge T = - \int_{\mathcal{X}} \partial \varphi \wedge \bar{\partial} \varphi \wedge T.
\]
\end{lemma}

A real $(1,1)$-form $\eta$ is called {positive} (resp. nonnegative) if the corresponding symmetric tensor defined by $(V,W) \mapsto \eta(V, JW)$ is positive (resp. nonnegative) definite for vector fields $V,W$. A real $(p,p)$ form is called positive (resp. nonnegative) if it is the sum of products of positive (resp. nonnegative) real $(1,1)$-forms. The integral of a nonnegative $(n,n)$-form is nonnegative. 

A Riemannian metric $g$ on $\mathcal{X}$ is called hermitian if $J$ is an orthogonal transformation with respect to $g$. A hermitian metric $g$ gives rise to a real $(1,1)$-form $\omega$ defined by $\omega(JV,W) = g(V, W)$. One says that a hermitian metric is K\"ahler if the corresponding $(1,1)$-form is $d$-closed. Conversely, we say a real $(1,1)$-form $\omega$ is compatible with $J$ if the equality $\omega(JV, JW) = \omega(V,W)$ holds for each pair of vector fields $V,W$. It is easily shown that the data of a K\"ahler metric is equivalent to the data of a $J$-compatible positive $d$-closed real $(1,1)$-form. By a K\"ahler orbifold $(\mathcal{X}, \omega)$ we mean an orbifold together with a choice of $J$-compatible positive $d$-closed real $(1,1)$-form $\omega$.

Many properties of K\"ahler manifolds, including Kodaira's $\ddb$-lemma, extend to the setting of K\"ahler orbifolds (see, for example, \cite{baily56}). 

\begin{lemma}[$\ddb$-lemma]
Let $(\mathcal{X},\omega)$ be a compact K\"ahler orbifold. If $\eta$ and $\eta'$ are two real $(1,1)$-forms in the same cohomology class, then there is a function $f : \mathcal{X} \to \mathbb{R}$ such that $\eta' = \eta + \sqrt{-1} \ddb f$.

\end{lemma}

A K\"ahler metric $g$ admits a local expression in the charts as 
\[
g_\alpha = (g_\alpha)_{j\bar{k}} dz_\alpha^j \otimes d\bar{z}_\alpha^k
\]
where $(g_\alpha)_{j\bar{k}} = g_\alpha(\partial/\partial z_\alpha^j, \partial/\partial \bar{z}_\alpha^k)$. The corresponding K\"ahler form $\omega$ admits local expression 
\[
\omega_\alpha = \sqrt{-1}(g_\alpha)_{j\bar{k}} dz_\alpha^j \wedge d\bar{z}_\alpha^k.
\]

The Ricci form $\text{Ric}(\omega)$ corresponding to a K\"ahler form $\omega$ is the $(1,1)$-form with local expression  
\[
\text{Ric}(\omega)_\alpha = -\sqrt{-1} \ddb \log \det ((g_\alpha)_{j \bar{k}}).
\]
It can be shown that the cohomology class of $\text{Ric}(\omega)$ does not depend on the particular K\"ahler metric $\omega$, and thus defines an invariant of the orbifold. The first Chern class $c_1(\mathcal{X})$ can be taken to be the real cohomology class determined by the form $\tfrac{1}{2\pi} \text{Ric}(\omega)$ for some choice of K\"ahler form $\omega$. We say that $c_1(\mathcal{X})$ is positive written $c_1(\mathcal{X}) > 0$ (resp. negative written $c_1(\mathcal{X}) < 0$) if $c_1(\mathcal{X})$ is represented by a positive (resp. negative) $(1,1)$-form.

For a K\"ahler metric $g$ on a complex orbifold there is a corresponding Laplacian $\Delta$, which is a second-order uniformly elliptic operator (see Section \ref{sec:elliptic}), defined locally on smooth functions $\varphi$  by  
\[
\Delta(\varphi_\alpha) = g_\alpha^{j\bar{k}} \partial_j \partial_{\bar{k}} \varphi_\alpha
\]
where $g_\alpha^{j\bar{k}}$ is the inverse of the matrix $(g_\alpha)_{j\bar{k}}$ representing the metric.  Alternatively, for a $(1,1)$-form with local expression $\eta_\alpha = \sqrt{-1} (\eta_\alpha)_{j\bar{k}} dz_\alpha^j \wedge d\bar{z}_\alpha^k$, we write $\text{tr}_\omega \eta$ for the smooth function with local expression 
\[
\text{tr}_\omega \eta = g^{j\bar{k}} (\eta_\alpha)_{j\bar{k}}. 
\] 
One can show that if $\varphi$ is a smooth function, then $\Delta \varphi$ satisfies 
\[
n \sqrt{-1} \ddb \varphi \wedge \omega^{n-1} = \text{tr}_\omega(\sqrt{-1} \ddb \varphi) \omega^n = \Delta \varphi \cdot \omega^n
\]
where $n$ is the complex dimension of the complex orbifold $\mathcal{X}$. 

For a smooth $\mathbb{C}$-valued function $\varphi$ on a K\"ahler orbifold $(\mathcal{X}, \omega)$, we have 
\[
|d\varphi|_g^2 = |\partial \varphi|_g^2 + |\bar{\partial}\varphi|_g^2.
\]
Moreover, if $\varphi$ is $\mathbb{R}$-valued, then $|\partial \varphi|_g^2 = |\bar{\partial}\varphi|_g^2$, and hence 
\[
|\partial \varphi|_g^2 = \frac{1}{2}|d\varphi|_g^2.
\]
With this convention, it follows that for a smooth $\mathbb{R}$-valued function $\varphi$ we have
\[
n\sqrt{-1} \partial \varphi \wedge \bar{\partial} \varphi \wedge \omega^{n-1} = |\partial \varphi|_g^2 \cdot \omega^n.
\]

\section{Elliptic operators on orbifolds}\label{sec:elliptic}

The goal of this section is to study some properties and estimates associated to linear elliptic second-order differential operators on a compact orbifold, to construct a Green's function of the complex Laplacian on a compact K\"ahler orbifold, and to establish some useful inequalities of Poincar\'e and Sobolev type. 

For a bounded domain $\Omega$ in $\mathbb{R}^n$, a natural number $k \in \mathbb{N}$, and a number $\alpha \in (0,1)$, recall the $C^{k,\alpha}$-norm of a function $f$ on $\Omega$ can be defined by 
\[
\norm{f}_{C^{k,\alpha}(\Omega)} = \sup_{|\mathbf{\ell}| \leqslant k} |\partial^{\mathbf{\ell}}f | + \sup_{|\mathbf{\ell}| = k} \sup_{\substack{x, y \in \Omega\\ x \ne y}} \frac{|\partial^{\mathbf{\ell}}f(x) - \partial^{\mathbf{\ell}}f(y)|}{|x - y|^\alpha} 
\]
where $\ell = (\ell_1, \ldots, \ell_n)$ is a multi-index and 
\[
\partial^\ell =  \frac{\partial}{\partial x^{\ell_1}} \cdots \frac{\partial}{\partial x^{\ell_n}}. 
\]
The space $C^{k,\alpha}(\Omega)$ is then the space of functions on $\Omega$ whose $C^{k,\alpha}$-norm is finite. 

A uniformly elliptic operator $L$ of second-order with smooth coefficients on $\Omega$ admits an expression of the form
\begin{align}\label{eqn:ellipticoperator}
L(f) = a^{ij}\partial_i\partial_j f + b^m\partial_m f + cf
\end{align}
where $a^{ij}, b^m, c$ are smooth functions and where we are using the Einstein summation convention. The uniform ellipticity implies that there are constants $\lambda, \Lambda > 0$ such that 
\[
\lambda |\xi|^2 \leqslant a^{ij}(x) \xi_i \xi_j \leqslant \Lambda |\xi|^2
\]
for each point $x \in \Omega$ and each vector $\xi \in \mathbb{R}^n$. The following a priori estimates for such operators are well-known (see, for example, \cite{gt15}). 

\begin{theorem}\label{thm:localelliptic}
Let $\Omega \subset \mathbb{R}^n$ be a bounded domain, let $\Omega' \subset \subset \Omega$ be a relatively compact subset, and let $\alpha \in (0,1)$ and $k \in \mathbb{N}$. Then there is a constant $C$ such that if $L(f) = h$ in $\Omega$, then we have 
\[
\norm{f}_{C^{k+2, \alpha}(\Omega')} \leqslant C(\norm{h}_{C^{k,\alpha}(\Omega)} + \norm{f}_{C^0(\Omega)}).
\]
Moreover, the constant $C$ only depends on $k,\alpha,$ the domains $\Omega$ and $ \Omega'$, the $C^{k,\alpha}$-norms of the coefficients of $L$, and the constants of ellipticity $\lambda, \Lambda$. Finally, it is enough only to assume that $f \in C^2(\Omega)$ so that $L(f) = g$ makes sense, and then it follows that $f \in C^{k+2, \alpha}(\Omega')$ whenever $L(f)$ and the coefficients of $L$ are in $C^{k,\alpha}(\Omega)$. 
\end{theorem}

These estimates can be extended to a compact orbifold $\mathcal{X}$ as follows. The H\"older spaces can be defined locally in orbifold charts: Covering $\mathcal{X}$ with orbifold charts, any tensor $T$ has a local expression as a bona fide tensor $T_\beta$ in these charts, and the $C^{k,\alpha}$-norm of $T$ can be defined to be the supremum of the $C^{k,\alpha}$-norms of these $T_\beta$. In particular, if $\mathcal{X}$ is compact, we may achieve that there are finitely many orbifold charts, and the supremum can be taken to be the maximum.  A linear second-order differential operator on an orbifold $\mathcal{X}$ is an operator which admits an expression as \eqref{eqn:ellipticoperator} in each orbifold chart. 

\begin{theorem}\label{thm:ellipticest}
Let $L$ be an elliptic second-order linear differential operator with smooth coefficients on a compact Riemannian orbifold $(\mathcal{X}, g)$, and fix $\alpha \in (0,1)$ and $k \in \mathbb{N}$. Then there is a constant $C$ such that if $L(f) = h$ in $\Omega$, then 
\[
\norm{f}_{C^{k+2, \alpha}(\mathcal{X})} \leqslant C(\norm{h}_{C^{k,\alpha}(\mathcal{X})} + \norm{f}_{C^0(\mathcal{X})}).
\]
Moreover, the constant $C$ only depends on $k,\alpha,$ the orbifold $\mathcal{X}$, the $C^{k,\alpha}$-norms of the coefficients of $L$, and the constants of ellipticity $\lambda, \Lambda$. Finally, it is enough to assume only that $f \in C^2(\mathcal{X})$, and then it follows that $f \in C^{k+2, \alpha}(\mathcal{X})$ whenever $L(f)$ and the coefficients of $L$ are in $C^{k,\alpha}(\mathcal{X})$. 
\end{theorem} 

\begin{proof}
As $\mathcal{X}$ is compact, there is a finite collection of orbifold charts $(\widetilde{U}_\beta, G_\beta, \pi_\beta)$ whose supports $U_\beta$ cover $X$. We may select relatively compact domains $\widetilde{U}_\beta' \subset \widetilde{U}_\beta$ invariant under the $G_\beta$ action satisfying the hypotheses of Theorem \ref{thm:localelliptic}, such that the images $\pi(\widetilde{U}_\beta')$ still cover $X$. We have a finite list of constants $C_\beta$ from Theorem \ref{thm:localelliptic} applied to each pair $(\widetilde{U}_\beta, \widetilde{U}_\beta')$, and taking the maximum gives a large constant $C$. Then an application of the previous theorem gives
\begin{align*}
\norm{f}_{C^{k+2, \alpha}(\mathcal{X})} &= \max_\beta \norm{f_\beta}_{C^{k+2, \alpha}(\widetilde{U}_\beta')} \\
& \leqslant \max_\beta C_\beta(\norm{h}_{C^{k,\alpha}(U_\beta)} + \norm{f}_{C^0(U_\beta)}) \\
& \leqslant C(\norm{h}_{C^{k,\alpha}(\mathcal{X})} + \norm{f}_{C^0(\mathcal{X})}),
\end{align*}
as desired. 
\end{proof}

Moreover, the mapping properties of such operators on compact orbifolds are well-understood, as demonstrated below. We require the following standard result in functional analysis which can be found, for example, in \cite[Proposition 2.9.7]{narasimhan68}.

\begin{lemma}
Let $E_1$ and $E_2$ be Banach spaces, and let $A_1, A_2 : E_1 \to E_2$ be bounded linear operators. Suppose that 
\begin{enumerate}
\item[(i)] $A_1$ is injective and the range of $A_1$ is closed;
\item[(ii)] $A_2$ is a compact operator.
\end{enumerate}
Then the range of $A_1 + A_2$ is closed in $E_2$. 
\end{lemma}

\begin{lemma}
Let $L$ be an elliptic second-oder operator with smooth coefficients on a compact orbifold $\mathcal{X}$. Then the range of $L : C^{k+2, \alpha}(\mathcal{X}) \to C^{k,\alpha}(\mathcal{X})$ is closed. 
\end{lemma}

\begin{proof}
Define two Banach spaces $E_1$ and $E_2$ by 
\begin{align*}
E_1 &= C^{k+2, \alpha}(\mathcal{X}) \\
E_2 &= C^{k,\alpha}(\mathcal{X}) \oplus C^{k, \alpha}(\mathcal{X}).
\end{align*}
Define two bounded linear operators $A_1, A_2 : E_1 \to E_2$ by 
\begin{align*}
A_1(f) &= (Lf, f) \\
A_2(f) &= (0, -f).
\end{align*}
The operator $A_1$ is injective because the second component is. Moreover, the range of $A_1$ is closed by the uniform bound 
\begin{align*}
\norm{f}_{C^{k+2, \alpha}(\mathcal{X})} &\leqslant C (\norm{Lf}_{C^{k,\alpha}(\mathcal{X})} + \norm{f}_{C^0(\mathcal{X})}) \\
&\leqslant C (\norm{Lf}_{C^{k,\alpha}(\mathcal{X})} + \norm{f}_{C^{k,\alpha}(\mathcal{X})}) = C \norm{A_1(f)}_{E_2}.
\end{align*}
Finally $A_2$ is compact by the Arzela-Ascoli theorem applied to the inclusion of $C^{k+2, \alpha}(\mathcal{X})$ into $C^{k,\alpha}(\mathcal{X})$. It follows from the previous lemma that $A_1 + A_2$ has closed range. But the range of $A_1 + A_2$ is identified with $\text{Ran}(L) \oplus \{0\}$, and hence $L$ has closed range. 
\end{proof}

\begin{theorem}\label{thm:ellipticisom}
Let $L$ be an elliptic second-oder operator with smooth coefficients on a compact Riemannian orbifold $(\mathcal{X}, g)$. Then $L$ is an isomorphism of Banach spaces 
\[
L : (\ker L)^\perp \cap C^{k+2, \alpha}(\mathcal{X}) \to (\ker L^*)^\perp \cap C^{k,\alpha}(\mathcal{X}),
\]
where $L^*$ denotes the adjoint of $L$ with respect to the indicated Banach norms. 
\end{theorem}

\begin{proof}
The restriction of the domain to $(\ker L)^\perp$ ensures that $L$ is injective and the restriction of the codomain to $(\ker L^*)^\perp = \overline{\text{Ran}(L)}$ ensures that $L$ is surjective by the previous lemma.  Hence $L$ is a bounded invertible linear operator between Banach spaces. We conclude that the inverse of $L$ is bounded as well from \cite[\S 23 Theorem 2]{kf70}, and therefore $L$ is an isomorphism of Banach spaces. 
\end{proof}

Additionally, using Sobolev spaces, one can obtain not only a Green's function of the complex Laplacian, but also inequalities Poincar\'e and Sobolev stype, as follows. 

For a compact Riemannian orbifold $(\mathcal{X}, g)$ and a number $p \geqslant 1$, let $L^p(\mathcal{X})$ denote the completion of the space of smooth functions under the norm 
\[
\norm{f}_p = \left(\int_{\mathcal{X}} |f|^p\: dV \right)^{1/p}.
\]
For any tensor $S$, we also use the same notation for the norm 
\[
\norm{S}_p = \left(\int_{\mathcal{X}} |S|_g^p\: dV \right)^{1/p}.
\]
For a nonnegative integer $k$, let $L_k^p(\mathcal{X})$ denote the completion of the space of smooth functions with respect to the norm 
\[
\norm{f}_{L_k^p}^p = \sum_{\ell \leqslant k} \norm{\nabla^\ell f}_p^p.
\]

\begin{remark}
With this above notation, for a compact K\"ahler orbifold $(\mathcal{X}, \omega)$ we find that the following coincide $\norm{df}_2 = \norm{\nabla f}_2 = \sqrt{2}\norm{\partial f}_2$.
\end{remark}

The following version of Rellich's lemma on orbifolds can be found in \cite{chiang90}. 

\begin{theorem}[Rellich]
For a compact Riemannian orbifold $(\mathcal{X},g)$, the inclusion $L_{k+1}^2(\mathcal{X}) \to L_{k}^2(\mathcal{X})$ is compact. 
\end{theorem}

If $L$ is a linear elliptic differential operator of second order, then $L$ defines a map $L : L_{k+2}^p(\mathcal{X}) \to L_k^p(\mathcal{X})$, and similarly to the case of H\"older norms, the mapping properties of $L$ with respect to these Sobolev norms are well-understood. 

\begin{theorem}\label{thm:ellipticisom2}
Let $L$ be an elliptic operator of second order on a compact Riemannian orbifold $(\mathcal{X}, g)$, and fix a number $p \geqslant 1$ and a number $k \in \mathbb{N}$. Then there is a constant $C$ such that for all smooth functions $f$ we have 
\[
\norm{f}_{L_{k+2}^p} \leqslant C (\norm{Lf}_{L_k^p} + \norm{f}_{L_k^p}).
\]
Moreover the constant $C$ depends only on $\mathcal{X}, g, L, k, p$. Finally, $L$ induces an isomorphism of Banach spaces 
\[
L : (\ker L)^\perp \cap L_{k+2}^p(\mathcal{X}) \to  (\ker L^*)^\perp \cap L_{k}^p(\mathcal{X}).
\]
\end{theorem}

Just as in the case of H\"older norms, there are local versions of these inequalities: For any relatively compact $\Omega' \subset \subset \Omega$, we have an inequality of the form 
\[
\norm{f}_{L_{k+2}^p(\Omega')} \leqslant C (\norm{Lf}_{L_k^p(\Omega)} + \norm{f}_{L_k^p(\Omega)}).
\]

In particular, the complex Laplacian $\Delta$ on a compact K\"ahler orbifold $(\mathcal{X}, \omega)$ is an elliptic operator of second order.  For a point $x \in X$, consider the real-valued function $\delta_x$ defined on smooth functions $\varphi$ by the rule
\[
\delta_x(\varphi) = \bar{\varphi} - \varphi(x)
\]
where $\bar{\varphi}$ denotes the average value of $\varphi$. Then we may regard $\delta_x$ as an element of the Hilbert space $L^2(\mathcal{X})$ (= $L_0^2(\mathcal{X})$) via the $L^2$-inner product (c.f. \cite{chiang90}).  Since $\delta_x$ vanishes on smooth functions with average value zero, it follows from Theorem \ref{thm:ellipticisom2} that there is an element $G_x$ of $L_{2}^2(\mathcal{X})$ such that $\Delta G_x = \delta_x$.  Such a function $G_x$ is called a Green's function associated to $\Delta$ and satisfies 
\[
\int_{\mathcal{X}} G_x \Delta \varphi \cdot \omega^n = \bar{\varphi} - \varphi(x). 
\]

\begin{theorem}[Poincar\'e inequality]\label{thm:poincare}
Let $(\mathcal{X}, g)$ be a compact K\"ahler orbifold. There is a constant $C$ depending only on $\mathcal{X}$ and $g$ such that if $\varphi$ is a smooth function with average value zero, then 
\[
\norm{\varphi}_2 \leqslant C \norm{\partial \varphi}_2.
\]
\end{theorem} 

\begin{proof}
For a function $\varphi \in L_1^2(\mathcal{X})$ satisfying $\norm{\varphi}_2 \ne 0$, let $R(\varphi)$ denote the Rayleigh quotient 
\[
R(\varphi) = \frac{\norm{\partial \varphi}_2}{\norm{\varphi}_2}. 
\]
Let $E$ denote the subspace of $L_1^2(\mathcal{X})$ consisting of all functions with average value zero, and let $\lambda$ denote the infimum 
\[
\lambda = \inf_{0 \ne \varphi \in E} R(\varphi).
\]

It suffices to show that $\lambda$ is nonzero, which we show. Suppose not. Let $\varphi_j$ denote a sequence of elements in $E$ satisfying $\lim_{j \to \infty} R(\varphi_j) = 0$. By scaling by $\norm{\varphi_j}_2^{-1}$, we may assume that each $\varphi_j$ satisfies $\norm{\varphi_j}_2 = 1$. It follows that the sequence $\varphi_j$ is uniformly bounded in $L_1^2(\mathcal{X})$. Rellich's lemma implies that, by passing to a subsequence, we can assume that the $\varphi_j$ converge in $L^2(\mathcal{X})$ to a function $\varphi \in L^2(\mathcal{X})$. This function must satisfy $\norm{\varphi}_2 = 1$. Moreover, for any smooth $(1,0)$-form $\psi$, if $\partial^*$ denotes the adjoint of $\partial$ with respect to the $L^2$-inner product induced by the K\"ahler metric $g$, then  
\begin{align*}
|\langle \varphi, \partial^* \psi \rangle_{L^2}| &= \lim_{j \to \infty}  |\langle \varphi_j, \partial^*\psi \rangle_{L^2}| \\
&= \lim_{j \to \infty}  |\langle \partial \varphi_j, \psi \rangle_{L^2}| \\
&\leqslant \lim_{j\to \infty}\norm{\partial \varphi_j}_2 \norm{\psi}_2 = 0
\end{align*}
so that $\partial \varphi = 0$ in the weak sense. The ellipticity of $\partial$ implies that $\varphi$ is actually smooth, and therefore a constant, with average value zero, and hence equal to zero. This contradicts the above deduction that $\norm{\varphi}_2 = 1$.  
\end{proof}

The following Sobolev inequality on bounded domains is well-known \cite{evans10}. 

\begin{lemma}[Local Sobolev inequality]
Let $\Omega$ be a bounded domain in $\mathbb{R}^n$. For a number $p \geqslant 1$, let $q$ denote the Sobolev conjugate satisfying $1/p +  1/q = 1/n$. Then there is a constant $C$ depending on $\Omega$ and $p$ such that for any smooth function $f$ with compact support in $\Omega$, we have 
\[
\norm{f}_{q}^2 \leqslant C(\norm{f}_p^2 + \norm{\nabla f}_p^2).
\]
In particular, if $\Omega \subset \mathbb{C}^n$ and $p = 2$, then $q = 2n/(n-1)$ and 
\[
\norm{f}_{\frac{2n}{n-1}}^2 \leqslant C(\norm{f}_2^2 + \norm{\partial f}_2^2)
\]
\end{lemma}

It follows that there is a similar type of inequality on compact K\"ahler orbifolds. 

\begin{theorem}[Sobolev inequality]\label{thm:sobolev}
Let $(\mathcal{X},g)$ be a compact K\"ahler orbifold of complex dimension $n$. There is a constant $C$ depending only on $\mathcal{X}$ and $g$ such that if $f$ is a smooth function then 
\[
\norm{f}_{\frac{2n}{n-1}}^2 \leqslant C (\norm{f}_2^2 + \norm{\partial f}_2^2).
\] 
\end{theorem}

\begin{proof}
Let $\varphi_\alpha$ be a partition of unity subordinate to the supports $U_\alpha$ of a finite collection of orbifold charts $\widetilde{U}_\alpha$ in an atlas.  The smooth function $\varphi_\alpha f$ is compactly supported in the support $U_\alpha$ of a chart $\widetilde{U}_\alpha$. The local Sobolev inequality in this chart $\widetilde{U}_\alpha$ implies the existence of a constant $C_\alpha$ such that 
\[
\norm{\varphi_\alpha f}_{\frac{2n}{n-1}}^2 \leqslant C_\alpha (\norm{\varphi_\alpha f}_2^2 + \norm{\partial (\varphi_\alpha f)}_2^2). 
\]
The triangle inequality implies that 
\[
|\partial (\varphi_\alpha f)| \leqslant |(\partial \varphi_\alpha)||f| + \varphi_\alpha |\partial f| \leqslant \left(\sup_{X} |\partial \varphi_\alpha| \right) |f| + |\partial f|. 
\]
Hence with the previous observation, we find an estimate of the form 
\[
\norm{\varphi_\alpha f}_{\frac{2n}{n-1}}^2 \leqslant C_\alpha (\norm{f}_2^2 + \norm{\partial  f}_2^2).
\]
Whence if $N$ denotes the number of charts and $C = N \cdot \sup_\alpha C_\alpha$, then 
\[
\norm{f}_{\frac{2n}{n-1}}^2 \leqslant \sum_\alpha \norm{\varphi_\alpha f}_{\frac{2n}{n-1}}^2 \leqslant C (\norm{f}_2^2 + \norm{\partial f}_2^2),
\]
as desired. 
\end{proof}

\section{Yau's theorem on orbifolds}\label{sec:main}

The goal of this section is to outline a proof of Theorem \ref{thm:main}. Before doing so, let us first demonstrate how Theorem \ref{thm:main} implies Theorem \ref{thm:calabi}. 

\begin{lemma}
Theorem \ref{thm:main} implies Theorem \ref{thm:calabi}. 
\end{lemma}

\begin{proof}
Because $R$ and $\text{Ric}(\omega)$ represent the same cohomology class, the $\ddb$-lemma gives a smooth function $\tilde{F}$ on $\mathcal{X}$ such that 
\[
R = -\sqrt{-1}\ddb \tilde{F} + \text{Ric}(\omega).
\]
If we let $C_1,C_2$ denote the positive quantities 
\[
C_1 = \int_\mathcal{X} e^{\tilde{F}} \omega^n \hspace{10mm} C_2 = \int_\mathcal{X} \omega^n
\] 
and we let $F$ denote the smooth function $F = \tilde{F} + \log(C_2/C_1)$, then $F$ satisfies the integration condition
\[
\int_\mathcal{X} e^{F}\omega^n  \int_\mathcal{X} \omega^n.
\]
By Theorem \ref{thm:main}, there is a smooth function $\varphi$ satisfying \eqref{eqn:MA}. 
The Ricci form of $\omega' = \omega + \sqrt{-1}\ddb \varphi$ then satisfies 
\[
\text{Ric}(\omega') = -\sqrt{-1} \ddb \log (e^F \omega^n) = -\sqrt{-1} \ddb F + \text{Ric}(\omega) = R. 
\]
Hence $\omega' = \omega + \sqrt{-1}\ddb \varphi$ is a solution to the Calabi conjecture.

Suppose that $\omega''$ is another solution. There is a $\varphi''$ such that $\omega'' = \omega + \sqrt{-1}\ddb \varphi''$. By assumption, the Ricci form of $\omega''$ satisfies 
\[
\text{Ric}(\omega'') = R = -\sqrt{-1}\ddb F + \text{Ric}(\omega),
\]
or equivalently, 
\[
-\sqrt{-1}\ddb(\log(\omega + \sqrt{-1}\ddb \varphi'')^n) = -\sqrt{-1}\ddb F - \sqrt{-1}\ddb\log(\omega^n).
\]
Rearranging gives 
\[
\sqrt{-1}\ddb\left(\log\frac{(\omega + \sqrt{-1}\ddb\varphi')^n}{\omega^n} - F\right) = 0.
\]
Integration by parts shows that there is a constant $C$ such that 
\[
\log\frac{(\omega + \sqrt{-1}\ddb\varphi'')^n}{\omega^n} - F = C, 
\]
 which implies that 
\[
(\omega + \sqrt{-1}\ddb\varphi'')^n = e^{F+C} \omega^n. 
\]
The fact that $\int_\mathcal{X} e^F \omega^n = \int_\mathcal{X} \omega^n$ implies that  
\[
e^C \int_\mathcal{X} e^{F} \omega^n = \int_\mathcal{X} e^{F+C}\omega^n = \int_\mathcal{X} \omega^n = \int_\mathcal{X} e^F \omega^n.
\]
Hence $C = 0$ and we conclude that $\varphi''$ is a solution to Theorem \ref{thm:main}. Thus $\varphi''$ and $\varphi'$ differ by a constant, and so $\omega''= \omega'$. This shows how the uniqueness in Theorem \ref{thm:calabi} follows from that of Theorem \ref{thm:main}. 
\end{proof}

Let us now move on to a proof of Theorem \ref{thm:main}.  First we deal with uniqueness. 

\begin{proposition}
If $\varphi, \varphi'$ are two smooth solutions to \eqref{eqn:MA}, then $\varphi$ and $\varphi'$ differ by a constant. 
\end{proposition}

\begin{proof}
Write $\omega_{\varphi} = \omega + \sqrt{-1} \ddb\varphi$. Then with this notation, we have
\[
0 = \omega_{\varphi}^n - \omega_{\varphi'}^n = \sqrt{-1}\ddb(\varphi - \varphi') \wedge T
\]
where 
\[
T = \sum_{k=0}^{n-1} \omega_{\varphi}^k \wedge \omega_{\varphi'}^{n-1-k}
\]
is a positive, closed $(n-1,n-1)$ form. Upon integrating by parts, we find 
\[
0 = \int_\mathcal{X} (\varphi - \varphi') (\omega_{\varphi}^n - \omega_{\varphi'}^n) = -  \int_{\mathcal{X}} \sqrt{-1} \partial(\varphi - \varphi') \wedge \bar{\partial}(\varphi - \varphi') \wedge T.
\]
The positivity of $T$ implies that the integral is nonnegative, and hence we must have $\partial(\varphi - \varphi') = 0$. Thus $\varphi - \varphi'$ is constant. 
\end{proof}

With the tools established in the previous sections, we can formulate a proof of Theorem \ref{thm:main} by following exactly the structure of a proof in the smooth setting. In particular, one approach is the following well-known continuity method. For completeness, we outline this approach now. 

The idea is to introduce a family of equations 
\begin{align}\label{eqn:*_t}
\begin{cases}
(\omega + \sqrt{-1}\ddb \varphi)^n = e^{tF} \omega^n \\
\omega + \sqrt{-1} \ddb \varphi \; \text{is a K\"ahler form}
\end{cases} \tag{$*_t$}
\end{align}
indexed by a parameter $t \in [0,1]$. The equation $(*_0)$ admits the trivial solution $\varphi \equiv 0$. Thus, if we can show that the set of such $t \in [0,1]$ for which $\eqref{eqn:*_t}$ admits a smooth solution is both open and closed, it will follow that we can solve $(*_1)$. For this endeavor, it suffices to prove the following 

\begin{proposition}\label{prop:suff} Fix an $\alpha \in (0,1)$.
\item[(i)] If \eqref{eqn:*_t} admits a smooth solution for some $t < 1$, then for all sufficiently small $\epsilon > 0$, the equation $(*_{t + \epsilon})$ admits a smooth solution as well. 
\item[(ii)] There is a constant $C > 0$ depending only on $\mathcal{X}, \omega, F,$ and $\alpha$ such that if $\varphi$ with average value zero satisfies \eqref{eqn:*_t} for some $t \in [0,1]$, then 
\begin{itemize}
\item $\norm{\varphi}_{C^{3,\alpha}(\mathcal{X})} \leqslant C$ and
\item $(g_{j\bar{k}} + \partial_j \partial_{\bar{k}}\varphi) > C^{-1} (g_{j\bar{k}})$, where $g_{j\bar{k}}$ are the components of $\omega$ in local coordinates of any chart and the inequality means that the difference of matrices is positive definite. 
\end{itemize}
\end{proposition}

Indeed Proposition \ref{prop:suff} is sufficient because we can obtain a solution to $(*_1)$ using the following lemma. 

\begin{lemma}
Assume Proposition \ref{prop:suff}. Then if $s$ is a number in $(0,1]$ such that we can solve $\eqref{eqn:*_t}$ for all $t < s$, then we can solve $(*_s)$. 
\end{lemma}

\begin{proof}
Let $t_i \in (0,1]$ be a sequence of numbers approaching $s$ from below. By assumption, this gives rise to a sequence of functions $\varphi_i$ satisfying 
\[
(\omega + \sqrt{-1}\ddb \varphi_i)^n = e^{t_iF} \omega^n.
\]
Proposition \ref{prop:suff} together with the Arzela-Ascoli theorem implies that after passing to a subsequence, we may assume that $\varphi_i$ converges in $C^{3,\alpha'}(\mathcal{X})$ to a function $\varphi$ for some $\alpha' < \alpha$. This convergence is strong enough that we find 
\[
(\omega + \sqrt{-1} \ddb \varphi)^n = e^{sF} \omega^n.
\]
Moreover, Proposition \ref{prop:suff} gives that the forms $(\omega + \sqrt{-1}\ddb \varphi_i)$ are bounded below by a fixed positive form $C^{-1} \omega$, so that $\omega + \sqrt{-1} \ddb \varphi$ is a positive form. 

It remains to show that $\varphi$ is smooth. In local coordinates, we find that $\varphi$ satisfies 
\[
\log \det(g_{b\bar{k}} + \partial_j \partial_{\bar{k}}\varphi) - \log \det (g_{j \bar{k}}) - sF = 0.
\]
Differentiating the equation with respect to the variable $z^\ell$ we have 
\[
(g_\varphi)^{j\bar{k}} \partial_j \partial_{\bar{k}} (\partial_\ell \varphi) = s \partial_\ell F +  \partial_{\ell} \log \det (g_{j\bar{k}}) - (g_\varphi)^{j\bar{k}} \partial_\ell g_{j \bar{k}}
\]
where $(g_\varphi)^{j\bar{k}}$ is the inverse of the matrix $(g_\varphi)_{j\bar{k}} = g_{j\bar{k}} + \partial_j \partial_{\bar{k}} \varphi$. We think of this equation as a linear elliptic second-order equation $L(\partial_{\ell}\varphi) = h$ for the function $\partial_{\ell}\varphi \in C^{2,\alpha'}(\mathcal{X})$. Because the function $h$ and the coefficients of $L$ belong to $C^{1, \alpha'}$, we conclude from Theorem \ref{thm:ellipticest} that $\partial_\ell\varphi$ belongs to $C^{3,\alpha'}$. Because $\ell$ was arbitrary, it follows that $\varphi$ belongs to $C^{4,\alpha'}$. Repeating this argument we obtain that $\varphi \in C^{5,\alpha'}$ and by induction, that $\varphi$ is actually smooth. This technique of considering the corresponding linear equation to obtain better regularity of solutions is called bootstrapping. 
\end{proof}

Let us now prove the first part of Proposition \ref{prop:suff}. 

\medskip 

\noindent \emph{Proof of Proposition \ref{prop:suff} (i)}. Let $B_1$ denote the Banach manifold consisting of those $\varphi \in C^{3,\alpha}(\mathcal{X})$ with average value zero and such that $\omega + \sqrt{-1} \ddb \varphi$ is a positive form. Let $B_2$ denote the Banach space consisting of those $\varphi \in C^{1,\alpha}(\mathcal{X})$ with average value zero. Define a mapping 
\begin{align*}
G : B_1 \times [0,1] &\longrightarrow B_2  \\
(\varphi, t) &\longmapsto \log \frac{(\omega + \sqrt{-1} \ddb \varphi)^n}{\omega^n} - tF.
\end{align*}
By assumption, we are given a smooth function $\varphi_t$ such that $G(\varphi_t, t) = 0$ and $\omega + \sqrt{-1} \ddb \varphi_t$ is a K\"ahler form. The partial derivative of $G$ in the direction of $\varphi$ at the point $(\varphi_t, t)$ is given by 
\[
DG_{(\varphi_t, t)}(\psi,0) = \frac{n \sqrt{-1} \ddb \psi \wedge \omega_t^{n-1}}{\omega_t^n} = \Delta_t \psi,
\]
where $\omega_t = \omega + \sqrt{-1} \ddb \varphi_t$ and $\Delta_t$ denotes the Laplacian with respect to $\omega_t$. Denote this partial derivative by the operator $L(\psi) = \Delta_t\psi$. 

The operator $L$ has trivial kernel. Indeed suppose $\psi$ satisfies $L(\psi) = 0$. Then integration by parts shows that 
\[
\int_{\mathcal{X}} \psi \Delta_t \psi \omega_t^n = - \int_{\mathcal{X}}  \sqrt{-1} \partial \psi \wedge \bar{\partial} \psi \wedge \omega_t^{n-1}
\]
and the positivity of $\omega_t^{n-1}$ implies that $\partial \psi = 0$. We conclude that $\psi$ is a constant, with average value zero, and hence equal to zero. 

Moreover two integrations by parts show that the operator $L$ is self-adjoint, and hence $L^*$ has trivial kernel as well. It follows from Theorem \ref{thm:ellipticisom} that $L$ is an isomorphism 
\[
L : C_0^{3,\alpha}(\mathcal{X}) \to C_0^{1,\alpha}(\mathcal{X})
\]
where $C_0^{k,\alpha}(\mathcal{X})$ denotes the subspace of functions in $C^{k,\alpha}(\mathcal{X})$ with average value zero. The implicit function theorem asserts that for $s$ sufficiently close to $t$, there are functions $\varphi_s$ in $C^{3,\alpha}_0(\mathcal{X})$ satisfying $G(\varphi_s, s) = 0$. Because $\varphi + \sqrt{-1} \ddb \varphi_t$ is a positive form, for $s$ close enough to $t$, we can ensure also that $\omega + \sqrt{-1} \ddb \varphi_s$ is also a positive form. Moreover, bootstrapping arguments similar to those described earlier show that $\varphi_s$ is actually smooth. \hfill $\Box$

\section{A uniform $C^{3,\alpha}$-estimate}\label{sec:est}

This section is devoted to proving Proposition \ref{prop:suff} (ii). There are many expositions of this statement in the smooth setting (see \cite{yau78, siu12, tian12, blocki12, szekelyhidi14}). Essentially any of these arguments can be modified to the orbifold setting, provided the necessary ingredients can be modified to the orbifold setting. We will outline a streamlined version of one of the arguments, which can be found in \cite{szekelyhidi14}, and we direct the reader to this resource for more details of the reasoning to follow. 

First we obtain a $C^0$-estimate using a method of Moser iteration. The argument in the smooth setting can be found in \cite{szekelyhidi14}, which follows an exposition due to \cite{blocki12}. For completeness, we outline the argument below, demonstrating how the tools of the previous sections (Green's function, Poincare inequality, Sobolev inequality) are used.  

\begin{lemma}[$C^0$-estimate]\label{lem:C^0}
There is a constant $C$ depending on $\mathcal{X}, \omega,$ and $F$ such that if $\varphi$ is a solution to \eqref{eqn:*_t} with average value zero, then 
\[
\norm{\varphi}_{C^0(\mathcal{X})} \leqslant C.
\]
\end{lemma}

\begin{proof}
Without loss of generality we may assume that $\omega$ is rescaled so that $\mathcal{X}$ has volume $1$. A uniform bound on the $C^0$-norm of $\varphi$ is the same as a uniform bound on $-\varphi$. Thus, to eliminate some minus signs in the calculations which follow, it is sufficient to assume that $\omega - \sqrt{-1} \ddb \varphi$ is a K\"ahler form and that $\varphi$ satisfies 
\[
(\omega - \sqrt{-1} \ddb \varphi)^n = e^{tF} \omega^n.
\]
For such $\varphi$ with average value zero, it suffices to give a bound 
\[
\sup_\mathcal{X} \varphi - \inf_{\mathcal{X}}\varphi \leqslant C.
\] 
Thus, shifting $\varphi$ by a constant, to prove the claim, it suffices to show that for solutions with $\inf \varphi = 1$, we have a bound 
\[
\sup_{\mathcal{X}} \varphi \leqslant C,
\]
which is what we will show. 

We first show that we have a uniform bound $\norm{\varphi}_1 \leqslant C$ on the $L^1$-norm of solutions $\varphi$. The form $\omega - \sqrt{-1}\ddb\varphi$ is positive so that after taking the trace with respect to $\omega$ we have 
\[
n - \Delta \varphi > 0.
\]
where $\Delta$ is the Laplacian with respect to $\omega$. Let $x$ be a point where $\varphi$ achieves its minimum, and let $G_x$ be a Green's function with respect to the Laplacian $\Delta$. We may assume that $G_x$ is integrable and, shifting by a constant, that $G_x$ is nonnegative. Then
\begin{align*}
\varphi(x) = \int_\mathcal{X} \varphi \cdot \omega^n - \int_\mathcal{X} G_x \Delta \varphi \cdot \omega^n  &\geqslant \int_\mathcal{X} \varphi \cdot \omega^n -n \int_{\mathcal{X}} G_x \cdot \omega^n \\
&\geqslant \int_\mathcal{X} \varphi \cdot \omega^n - C. 
\end{align*}
It follows that we have a uniform estimate $\norm{\varphi}_1 \leqslant C$ as desired. 

We next show that we have a uniform estimate $\norm{\varphi}_2 \leqslant C$. If we write $\omega_{\varphi} = \omega - \sqrt{-1} \ddb \varphi$, then we compute that 
\begin{align*}
\int_{\mathcal{X}} \varphi(\omega_\varphi^n - \omega^n) = \int_{\mathcal{X}} \sqrt{-1} \partial \varphi \wedge \bar{\partial} \varphi \wedge T
\end{align*}
where $T$ is the positive $(n-1,n-1)$ form 
\[
T = \sum_{k=0}^{n-1} \omega^k \wedge \omega_\varphi^{n-1-k}.
\]
It follows that 
\[
\int_{\mathcal{X}} \varphi(\omega_\varphi^n - \omega^n) \geqslant \int_{\mathcal{X}} \sqrt{-1}\partial \varphi \wedge \bar{\partial}\varphi \wedge \omega^{n-1} 
\]
From the observation that $\omega_\varphi^n - \omega^n = (e^{tF}-1)\omega^n$ together with the estimate of the previous paragraph, we find that 
\[
\norm{\partial \varphi}_2^2 \leqslant C.
\]
The Poincar\'e inequality (Theorem \ref{thm:poincare}) then implies that 
\[
\int_{\mathcal{X}} (\varphi - \norm{\varphi}_1)^2 \omega^n \leqslant C \norm{\partial \varphi}_2^2 \leqslant C.
\]
Hence our bound from the previous paragraph implies a bound $\norm{\varphi}_2 \leqslant C$. 

Finally it is routine to use a technique called Moser iteration to establish a uniform bound $\sup_\mathcal{X} \varphi \leqslant C \norm{\varphi}_2$, which will complete the proof. For $p \geqslant 2$, we have 
\[
\int_{\mathcal{X}} \varphi^{p-1}(\omega_\varphi^n - \omega^n) = \frac{4(p-1)}{p^2} \int_{\mathcal{X}} \sqrt{-1} \partial \varphi^{p/2} \wedge \bar{\partial} \varphi^{p/2} \wedge T,
\] 
which implies 
\[
\int_{\mathcal{X}} \varphi^{p-1}(\omega_\varphi^n - \omega^n) \geqslant \frac{4(p-1)}{p^2} \int_\mathcal{X} \sqrt{-1} \partial \varphi^{p/2} \wedge \bar{\partial} \varphi^{p/2} \wedge \omega^{n-1}.
\]
We deduce that 
\[
\norm{\partial \varphi^{p/2}}_2^2 \leqslant Cp \norm{\varphi}_{p-1}^{p-1}
\]
for some constant $C$ independent of $p$.  The Sobolev inequality (Theorem \ref{thm:sobolev}) applied to $\varphi^{p/2}$ together with this estimate gives that 
\begin{align*}
\norm{\varphi}_{\frac{np}{n-1}}^p = \norm{\varphi^{p/2}}^2_{\frac{2n}{n-1}} &\leqslant C\left(\norm{\varphi^{p/2}}_2^2 + \norm{\partial \varphi^{p/2}}_2^2\right) \\
&\leqslant C\left(\norm{\varphi}_p^p + Cp \norm{\varphi}_{p-1}^{p-1}\right) \\
&\leqslant Cp\norm{\varphi}_p^p.
\end{align*}
If we write $p_k = (n/(n-1))^kp$, then we find 
\begin{align*}
\norm{\varphi}_{p_k} &\leqslant (Cp_{k-1})^{1/(p_k-1)}\norm{\varphi}_{p_k-1} \\
&\leqslant \norm{\varphi}_p \prod_{i=0}^{k-1} (Cp_i)^{1/p_i} \\
&\leqslant \norm{\varphi}_p \prod_{i=0}^{\infty}(Cp_i)^{1/p_i}.
\end{align*}
If we set $p = 2$ and let $k \to \infty$, then we find that 
\[
\sup_\mathcal{X} \varphi \leqslant C \norm{\varphi}_2,
\]
and the estimate from the previous paragraph on $\norm{\varphi}_2$ gives the desired bound. 
\end{proof}

The following lemma can be proved by a local calculation, which uses the Cauchy-Schwarz inequality twice and which can be found, for example, in \cite[Lemma 3.7]{szekelyhidi14}. 

\begin{lemma}
There is a constant $C$ depending on $\mathcal{X}$ and $\omega$ such that if $\varphi$ is a solution to \eqref{eqn:*_t} with average value zero, then
\[
\hat{\Delta} \log \textnormal{tr}_\omega \omega_\varphi \geqslant - C \textnormal{tr}_{\omega_\varphi}\omega - \frac{g^{j\bar{k}} \hat{R}_{j\bar{k}}}{\textnormal{tr}_\omega \omega_\varphi}
\]
where $\hat{\Delta}$ is the Laplacian with respect to $\omega_\varphi$ and $\hat{R}_{j\bar{k}}$ is the Ricci curvature of $\omega_\varphi$. 
\end{lemma}

The $C^2$-estimate then follows directly from this lemma together with the $C^0$-estimate, again by a local computation which uses only rudimentary tools such as the Cauchy-Schwarz inequality and which can be found again in \cite[Lemma 3.8]{szekelyhidi14}.

\begin{lemma}[$C^2$-estimate]\label{lem:uniformequiv}
There is a constant $C$ depending on $\mathcal{X}, \omega, F$ such that a solution $\varphi$ of \eqref{eqn:*_t} with average value zero satisfies 
\[
C^{-1}(g_{j\bar{k}}) < (g_{j\bar{k}} + \partial_j \partial_{\bar{k}}\varphi) < C (g_{j\bar{k}}). 
\]
\end{lemma}

Let $S$ denote the tensor given by the difference of Levi-Civita connections $S = \hat{\nabla} - \nabla$. Note that $S$ depends on the third derivatives of $\varphi$. So if $|S|$ denotes the norm of $S$ with respect to the metric $\omega_\varphi$, the fact that the metric $g_{j\bar{k}}$ is uniformly equivalent to the metric $g_{j\bar{k}} + \partial_j \partial_{\bar{k}}\varphi$ implies that a bound on $|S|$ gives a $C^3$-bound on $\varphi$.

\begin{lemma}[$C^3$-estimate]\label{lem:C^3}
There is a constant $C$ depending on $\mathcal{X}, \omega, F$ such that if $\varphi$ is a solution to \eqref{eqn:*_t} with average value zero, then $|S| \leqslant C$, where $|S|$ is the norm of $S$ computed with respect to the metric $\omega_\varphi$. 
\end{lemma}

\begin{proof}
Again local computations and rudimental local identities from complex geometry can be used first to obtain estimates of the form 
\[
\hat{\Delta}|S|^2 \geqslant - C|S|^2 - C
\]
and  
\[
\hat{\Delta} \text{tr}_\omega \omega_\varphi \geqslant -C + \epsilon |S|^2
\]
where $C$ denotes a large constant and $\epsilon$ a small one. We direct the reader again to \cite[Lemma 3.9 and Lemma 3.10]{szekelyhidi14} for local proofs of these estimates. 

It follows that we may choose a large constant $A$ such that 
\[
\hat{\Delta}(|S|^2 + A \text{tr}_\omega \omega_\varphi) \geqslant |S|^2 - C
\]
Suppose that $|S|^2 + A \text{tr}_\omega \omega_\varphi$ achieves its maximum at $x \in X$. Then in a local chart around $x$ we have a local inequality of the form 
\[
0 \geqslant |S|^2(x) - C
\]
so that $|S|^2(x) \leqslant C$. At any other point $y \in X$, this bound together with the $C^2$-estimate imply any estimate of the form 
\[
|S|^2(y) \leqslant (|S|^2 + A \text{tr}_\omega \omega_\varphi)(y) \leqslant (|S|^2 + A \text{tr}_\omega \omega_\varphi)(x) \leqslant C.
\]
This is what we wanted. 
\end{proof}

We are now able to complete a proof of Proposition \ref{prop:suff}. 
\medskip

\noindent \emph{Proof of Proposition \ref{prop:suff} (ii)}. Lemma \ref{lem:uniformequiv} shows that the metric $\omega_\varphi$ is uniformly equivalent to the metric $\omega$. Lemma \ref{lem:C^3} implies that we have a uniform bound of the form $\norm{\varphi}_{C^3(\mathcal{X})} \leqslant C$, from which it follows that we have a uniform bound of the form $\norm{\varphi}_{C^{2,\alpha}(\mathcal{X})} \leqslant C$.  Bootstrapping arguments together with Theorem \ref{thm:ellipticest} and Lemma \ref{lem:C^0} imply that we actually have a uniform bound $\norm{\varphi}_{C^{3,\alpha}(\mathcal{X})} \leqslant C$, as desired. \hfill $\Box$

\section{K\"ahler-Einstein metrics for orbifolds with $c_1(\mathcal{X}) < 0$}\label{sec:KE}

The goal of this section is to outline briefly a proof of Theorem \ref{thm:KE}.  One way to proceed is by showing that the K\"ahler-Einstein condition is equivalent to a Monge-Amp\`ere equation that is only slightly different than \eqref{eqn:MA}, and then use the same technique of the continuity method to solve this slightly modified equation. 

Indeed, let $\omega$ be any K\"ahler metric in $-2\pi c_1(\mathcal{X})$. The Ricci form $\text{Ric}(\omega)$ belongs to the class $2\pi c_1(\mathcal{X})$, so by the $\ddb$-lemma there is a smooth function $F$ such that 
\[
\text{Ric}(\omega) = -\omega + \sqrt{-1}\ddb F.
\] 
Any other K\"ahler metric can be written as $\omega_\varphi = \omega + \sqrt{-1}\ddb\varphi$ for a smooth function $\varphi$ on $\mathcal{X}$. The Ricci form of $\omega_\varphi$ satisfies 
\[
\text{Ric}(\omega_\varphi) = \text{Ric}(\omega) - \sqrt{-1} \ddb \log \frac{\omega_\varphi^n}{\omega^n}
\]
and so the K\"ahler-Einstein condition $\text{Ric}(\omega_\varphi) = - \omega_\varphi$ reduces to 
\[
-\sqrt{-1} \ddb \varphi = \sqrt{-1} \ddb F - \sqrt{-1} \ddb \log \frac{\omega_\varphi^n}{\omega^n}. 
\]
For this equation to be true, it suffices to solve the Monge-Amp\`ere equation \eqref{eqn:MA2}.

To solve equation \eqref{eqn:MA2}, one can use a continuity method as in the case of \eqref{eqn:MA}. Indeed, one can introduce the family of equations 
\begin{align}\label{eqn:*_s}
\begin{cases}
(\omega + \sqrt{-1}\ddb \varphi)^n = e^{sF + \varphi} \omega^n \\
\omega + \sqrt{-1} \ddb \varphi \; \text{is a K\"ahler form}
\end{cases} \tag{$*_s$}
\end{align}
indexed by a parameter $s \in [0,1]$ and show that the set of such $s$ for which \eqref{eqn:*_s} admits a smooth solution is both open and closed. The openness follows from the implicit function theorem as above, with the only modification being that the linearized operator at $s$ is given by 
\[
\psi \mapsto \Delta_s \psi - \psi.
\]
The closedness follows from appropriate $C^0$-, $C^2$-, and $C^3$-estimates for solutions of \eqref{eqn:*_s}, which can be obtained from only very slight modifications of the arguments for the corresponding estimates for solutions of \eqref{eqn:*_t}. Moreover, the case of the $C^0$-estimate is even easier for solutions of \eqref{eqn:*_s}, as one can argue using the maximum principle (see \cite[Lemma 3.6]{szekelyhidi14} for a proof in the nonsingular case).

\section{Examples of Calabi-Yau orbifolds}\label{sec:rf}

Theorem \ref{thm:calabi} produces Ricci-flat K\"ahler metrics on orbifolds with $c_1(\mathcal{X}) = 0$ as a real cohomology class in $H^2(\mathcal{X}, \mathbb{R})$. In this section, we give examples of such orbifolds, which we call Calabi-Yau orbifolds. 

Previously we had defined the first Chern class of a K\"ahler orbifold $(\mathcal{X}, \omega)$ using the Ricci form $\text{Ric}(\omega)$. Alternatively, the first Chern class can be defined as a real cohomology class using connections and Chern-Weil theory as usual. In particular, the square of a unitary connection $\nabla$ on $\mathcal{X}$ corresponds to a mapping $F_\nabla$ taking $(1,0)$-tensors to $(1,2)$-tensors which is linear over the ring of smooth complex-valued functions on $\mathcal{X}$. The curvature $F_\nabla$ determines a characteristic polynomial in $t$
\[
\det\left(\text{id} - \frac{1}{2\pi \sqrt{-1}} F_\nabla t\right) = \sum_{k=0}^{n} f_k(\mathcal{X}) t^k
\] 
where $f_k(\mathcal{X})$ corresponds to a \emph{real} $2k$-form on $\mathcal{X}$ whose complex dimension we have denoted by $n$. The cohomology class determined by $f_k(\mathcal{X})$ independent of the choice of unitary connection and defines a real cohomology class $c_k(\mathcal{X}) \in H^{2k}(\mathcal{X}, \mathbb{R})$ called the $k$th Chern class.  

There is also a way to define the Chern classes as integral cohomology classes $c_k(\mathcal{X}) \in H^{2k}(\mathcal{X}, \mathbb{Z})$ (see \cite{alr07}). For our purposes, it is enough to know that the integral first Chern class $c_1(\mathcal{X})$ vanishes in $H^2(\mathcal{X}, \mathbb{Z})$ if and only if the sheaf $\omega_{\mathcal{X}}$ of germs of $n$-forms is isomorphic to the trivial invertible sheaf $\mathcal{O}_{\mathcal{X}}$ of germs of holomorphic functions. If $c_1(\mathcal{X})$ vanishes as an integral cohomology class, then it must also vanish as a real cohomology class. However, the converse is not true in general, as we will see in Example \ref{ex:eo} below. 

In the special case that $\dim_{\mathbb{C}}\mathcal{X} = 1$, the condition that $c_1(\mathcal{X})$ vanish as a real cohomology class is equivalent to the condition that the degree 
\[
\deg(\omega_{\mathcal{X}}) = - \langle c_1(\mathcal{X}), [\mathcal{X}] \rangle = -\int_{\mathcal{X}} c_1(\mathcal{X})
\]
of the sheaf $\omega_{\mathcal{X}}$ vanishes, where $[\mathcal{X}] \in H_{2}(\mathcal{X}, \mathbb{R})$ denotes the fundamental class determined by a choice of orientation on $\mathcal{X}$.  Let us use this observation first to classify all Calabi-Yau orbifold Riemann surfaces, which we call elliptic orbifolds. 

\begin{example}\label{ex:eo}
Let $\mathcal{X}$ be a connected closed orbifold Riemann surface with finitely many stacky points $p_1, \ldots, p_n$ in the underlying space $X$ . (We assume that $\mathcal{X}$ contains at least one stacky point, or equivalently, $n > 0$.) For each $i$, let $m_i > 1$ denote the size of the stabilizer group of $p_i$.  We may assume that we have ordered the points so that $m_1 \geqslant \cdots \geqslant m_n$. The goal of this example is to show that the statement $c_1(\mathcal{X}) = 0$ as a real cohomology class can be stated in terms of the data $n, m_1, \ldots, m_n$, and will hence give a finite list of possibilities. 

Let $\mathcal{Y}$ denote the complex manifold whose complex structure is determined by the atlas represented by the open subsets $\widetilde{U} \subset \mathbb{C}$ from the charts of $\mathcal{X}$ together with the transition functions determined by the embeddings of these charts. Then $\mathcal{Y}$ is a connected closed smooth Riemann surface of genus $g$. In particular, $\mathcal{Y}$ is an effective orbifold in which every group in every orbifold chart is trivial. Let $\pi : \mathcal{X} \to \mathcal{Y}$ denote the corresponding canonical smooth map of effective orbifolds, which we study presently. 

Away from the stacky points, the map $\pi$ is an isomorphism of effective orbifolds. More precisely, denote by $q_i$ the point in $\mathcal{Y}$ corresponding to $p_i$ in $X$. Let $U$ denote the subset $U = X \setminus \{p_1, \ldots, p_n\}$ together with the orbifold structure determined by $\mathcal{X}$, and let $V$ denote the complex submanifold of $\mathcal{Y}$ given by $V = Y \setminus \{q_1, \ldots, q_n\}$ where $Y$ is the underlying space of $\mathcal{Y}$. Then the restriction 
\[
\pi|_{U} : U \to V
\]
is an isomorphism of effective orbifolds. 

Near the point $p_i$, however, the map $\pi$ can be described as follows. If $w$ is a local coordinate on $\mathcal{Y}$ near $q_i$ and if $z$ is a local coordinate on $\mathcal{X}$ near $p_i$, then $\pi^*w = z^{m_i}$.  It follows that if $\mathcal{O}_{\mathcal{Y}}(q_i)$ denotes the locally free sheaf corresponding to the divisor $q_i$ in $\mathcal{Y}$, then upon pulling back to $\mathcal{X}$ we obtain $\pi^*(\mathcal{O}_{\mathcal{Y}}(q_i)) = \mathcal{O}_{\mathcal{X}}(m_ip_i)$. Since the degree of the line bundle $\mathcal{O}_{\mathcal{Y}}(q_i)$ is $1$ and pulling back is compatible with taking degree, we find that 
\[
\deg (\mathcal{O}_{\mathcal{X}}(m_ip_i)) = \deg(\pi^*(\mathcal{O}_{\mathcal{Y}}(q_i))) = \deg(\mathcal{O}_{\mathcal{Y}}(q_i) ) = 1.
\]
We conclude that we must have 
\[
\deg(\mathcal{O}_{\mathcal{X}}(p_i)) = \frac{1}{m_i}. 
\]

If $\omega_{\mathcal{X}}$ denotes the coherent sheaf of germs of one-forms on $\mathcal{X}$, then we claim that 
\begin{align}\label{eqn:dualizing}
\pi^*\omega_{\mathcal{Y}} = \omega_{\mathcal{X}} \left(\sum_{i=1}^n(1-m_i)p_i\right).
\end{align}
Indeed, away from the stacky points, the map $\pi$ is an isomorphism, so we have 
\[
\left.\pi^*\omega_{\mathcal{Y}}\right|_{{U}} = \left.\omega_{\mathcal{X}}\right|_{U} = \left.\omega_{\mathcal{X}} \left(\sum_{i=1}^n(1-m_i)p_i\right)\right|_U
\]
where the last equality follows because $p_i \notin U$. 
On the other hand, near the point $p_i$, the projection map $\pi$ satisfies 
\[
\pi^*(dw) = d(z^{m_i}) = m_i z^{m_i-1}dz. 
\]
If $U_i$ is a neighborhood of $p_i$ satisfying $U_i \cap \{p_1, \ldots, p_n\} = p_i$, then the previous equality shows that 
\[
\left.\pi^*\omega_{\mathcal{Y}}\right|_{{U_i}}  = \left.\omega_{\mathcal{X}} \left(\sum_{i=1}^n(1-m_i)p_i\right)\right|_{U_i}.
\]
The claim now follows. 

Taking the degree of \eqref{eqn:dualizing}, we find that 
\[
2g -2 = \deg(\omega_{\mathcal{X}}) + \sum_{i=1}^n \frac{1-m_i}{m_i} = \deg(\omega_{\mathcal{X}}) -n + \sum_{i=1}^n \frac{1}{m_i}. 
\]
And hence the degree of the canonical sheaf $\omega_{\mathcal{X}}$ satisfies 
\[
\deg(\omega_{\mathcal{X}}) = 2g - 2 + n - \sum_{i=1}^n \frac{1}{m_i}. 
\]
Therefore, the condition $c_1(\mathcal{X}) = 0$ as a real cohomology class is equivalent to the equation
\begin{align}\label{eqn:cy1}
2 -2g = \sum_{i=1}^n \left(1 - \frac{1}{m_i}\right). 
\end{align}
However, this condition is not enough to ensure that the first Chern class vanishes as an \emph{integral} cohomology class, and in fact, even with this condition, the first Chern class is \emph{never} zero in $H^2(\mathcal{X}, \mathbb{Z})$ because the integral first Chern class restricts to a generator $c_1(\mathcal{X})|_{p_i} \in H^2(p_i, \mathbb{Z}) \simeq \mathbb{Z}/m_i \mathbb{Z}$ of the integral cohomology group of each stacky point.  Nevertheless, if $m$ is the least common multiple of $m_1, \ldots, m_n$, then $\omega_{\mathcal{X}}^{\otimes m}$ is trivial, so that the multiple $mc_1(\mathcal{X})$ \emph{is} zero as an integral cohomology class. 

Now let us determine the possibilities for $n, m_1, \ldots, m_n$. Each term on the right hand side of \eqref{eqn:cy1} is at least $1/2$ and less than $1$, so we find that 
\[
\frac{n}{2} \leqslant 2 - 2g < n.
\]
Since $n \geqslant 1$, we conclude that $g = 0$, and thus there are two cases for $n$: either $n = 3$ or $n = 4$. 
\begin{itemize}
\item Suppose that $n = 4$. Then equation \eqref{eqn:cy1} implies that we have 
\[
\frac{1}{m_1} + \frac{1}{m_2} + \frac{1}{m_3} + \frac{1}{m_4} = 2.
\]
There is only one possibility 
\begin{enumerate}
\item[(i)] $m_1 = m_2 = m_3 = m_4 = 2$.  
\end{enumerate}
\item Suppose that $n = 3$. Then equation \eqref{eqn:cy1} implies that we have 
\[
\frac{1}{m_1} + \frac{1}{m_2} + \frac{1}{m_3} = 1.
\]
We also have that $m_1 \geqslant m_2 \geqslant m_3 \geqslant 2$. There are three possibilities 
\begin{enumerate}
\item[(ii)] $m_1 = m_2 = 4, m_3 = 2$;
\item[(iii)] $m_1 = 6, m_2 = 3, m_3 = 2$;
\item[(iv)] $m_1 = m_2 = m_3 = 3$. 
\end{enumerate}

\end{itemize}

Each of these cases (i) through (iv) can be realized explicitly as a quotient of an elliptic curve. Indeed for a complex number $\tau$ in the upper half plane $\{z \in \mathbb{C} : \text{Im}(z) > 0\}$, let $E_\tau$ denote the smooth elliptic curve 
\[
E_\tau = \mathbb{C}/(\mathbb{Z} + \mathbb{Z}\tau).
\]
The flat K\"ahler metric on $\mathbb{C}$ descends to a flat K\"ahler metric on $E_\tau$. If $\Gamma$ is an finite group acting holomorphically and isometrically on $E_\tau$, then the flat K\"ahler metric on $E_\tau$ descends to a flat K\"ahler metric (hence Ricci flat) metric on the global quotient orbifold $[E_\tau/\Gamma]$. 

\begin{enumerate}
\item[(i)] If $\tau$ is any element of the upper half plane, then the group $\Gamma = \{\pm1\}$ acts on $\mathbb{C}$ in such a way that the action descends to the quotient $E_\tau$. If $[0]$ denotes the point in $E_\tau$ corresponding to $0 \in \mathbb{C}$, then $[0]$ is fixed by every element of $\Gamma$. The same is also true of the points $[1/2], [\tau/2]$, and $[(1 + \tau)/2]$. It follows that the orbifold $[E_\tau/\Gamma]$ is an elliptic orbifold. Such an orbifold is called a ``pillowcase'' in the literature. 
\item[(ii)] Suppose in particular that $\tau = \sqrt{-1}$. The group $\Gamma =  \{\pm1,\pm \tau\} \simeq \mathbb{Z}_4$ acts on $\mathbb{C}$ in such a way that the action descends to one on $E_\tau$. The points $[0]$ and $[(1 + \sqrt{-1})/2]$ are fixed by every element of $\Gamma$, and hence have a stabilizer group of order $4$.  The point $[1/2]$ is fixed by the subgroup of size two consisting of $\{1, -1\}$. It follows that the orbifold $[E_\tau/\Gamma]$ corresponds to the elliptic orbifold $\mathbb{P}^1_{4,4,2}$. 
\item[(iii)] Suppose that $\tau = e^{\pi \sqrt{-1}/3}$. The rotations generated by $-1$ and $e^{2\pi\sqrt{-1}/3}$ act in a well-defined way on $E_\tau$ so that the group $\Gamma = \mathbb{Z}_6 \simeq \mathbb{Z}_2 \times \mathbb{Z}_3$ acts on $E_\tau$. The points $[0], [(\tau + 1)/3]$, and $[(\tau + 1)/2]$ have stabilizers of orders $6, 3,$ and $2$ respectively. It follows that $[E_\tau/\Gamma]$ corresponds to the elliptic orbifold  $\mathbb{P}^1_{6,3,2}$. 
\item[(iv)] Suppose again that $\tau = e^{\pi \sqrt{-1}/3}$. The group $\mathbb{Z}_3$ acts on $\mathbb{C}$ by rotations generated by $\tau^2 = e^{2\pi \sqrt{-1}/3}$. This action descends to a $\mathbb{Z}_3$-action on $E_\tau$. The points $[0], [(1 + \tau)/3],$ and $[2(1 + \tau)/3]$ are distinct points in the quotient with stabilizers of order $3$. It follows that $[E_\tau/\Gamma']$ corresponds to the elliptic orbifold $\mathbb{P}^1_{3,3,3}$.   
\end{enumerate}
This completes our discussion of elliptic orbifolds.
\end{example}

For examples of higher dimensional Calabi-Yau orbifolds, one can consider complete intersections in weighted projective space. In this case, one can construct examples where the first Chern class vanishes as an \emph{integral} cohomology class. 

\begin{example}
For positive integers $q_0, \ldots, q_n$ satisfying $\text{gcd}(q_0, \ldots, q_n) = 1$, let $\mathbb{CP}[q_0, \ldots, q_n]$ denote $n$-dimensional weighted projective space. Recall that as a topological space
\[
\mathbb{CP}[q_0, \ldots, q_n] = (\mathbb{C}^{n+1} \setminus \{0\})/\sim
\]
where $\sim$ is the equivalence relation defined by $(z_0, \ldots, z_n) \sim (w_0, \ldots, w_n)$ if and only if there is a nonzero $\lambda \in \mathbb{C}^*$ such that $z_i = \lambda^{q_i}w_i$ for each $i$. 

It is useful to view $\mathbb{CP}[q_0, \ldots, q_n]$ as a toric variety.  Let $N$ be a lattice spanned by vectors $u_0, \ldots, u_n$ satisfying the relation $q_0u_0 + \cdots + q_nu_n = 0$, and let $\Sigma$ be the fan of all cones generated by proper subsets of $\{u_0, \ldots, u_n\}$. Then $\mathbb{CP}[q_0, \ldots, q_n]$ is the toric variety $X_\Sigma$ corresponding to the fan $\Sigma$. Let $D_i$ denote the torus-invariant divisor in $\mathbb{CP}[q_0, \ldots, q_n]$ corresponding to the one-dimensional cone spanned by $u_i$.  According to \cite[Exercise 4.1.5]{cls11} the class group $\text{Cl}(X_\Sigma)$ can be identified with $\mathbb{Z}$ in such a way that the divisor $\sum_i a_i D_i$ determines the element $\sum_{i}a_i q_i$ of $\mathbb{Z} \simeq \text{Cl}(X_\Sigma)$.  The divisor $-D_0 - \cdots - D_n$ corresponds to the dualizing sheaf of $X_\Sigma$ which in term corresponds to the element $-q_0 - \cdots - q_n \in \mathbb{Z} \simeq \text{Cl}(X_\Sigma)$.

Let $\mathbb{C}[x_0, \ldots, x_n]$ denote the polynomial ring corresponding to $\mathbb{CP}[q_0, \ldots, q_n]$, where each $x_i$ has degree $q_i$. A polynomial $F$ in $\mathbb{C}[x_0, \ldots, x_n]$ has degree $d$ if each monomial $x^\alpha$ appearing in $F$ satisfies $\alpha \cdot (q_0, \ldots, q_n) = d$. Accordingly, a polynonial $F$ of degree $d$ corresponds to a global section of the sheaf corresponding to $d \in \mathbb{Z} \simeq \text{Cl}(X_\Sigma)$, so that the hypersurface $\{F = 0\}$ in $\mathbb{CP}[q_0, \ldots, q_n]$ has normal sheaf corresponding to $d \in \mathbb{Z} \simeq \text{Cl}(X_\Sigma)$. 

Let $F_1, \ldots, F_s$ be homogeneous polynomials in $\mathbb{C}[x_0, \ldots, x_n]$ of degrees $d_1, \ldots, d_s$ respectively. Then the subset of weighted projective space given by $Y = \{F_1 = \cdots = F_s = 0\}$ is a complete intersection subvariety. If $Y$ has at most quotient singularities, then $Y$ is a complex orbifold. Remarks in the previous paragraph imply that the top power of the normal sheaf of $Y$ is the sheaf corresponding to $d_1 + \cdots + d_s \in \mathbb{Z} \simeq \text{Cl}(X_\Sigma)$.  Because $Y$ and $X$ are Cohen-Macaulay, adjunction still holds, which gives that 
\[
\omega_Y = \omega_X(d_1 + \cdots + d_s).
\]
We find that $\omega_Y$ is trivial if and only if 
\begin{align*}\label{eqn:c_1=0}
d_1 + \cdots + d_s = q_0 + \cdots + q_n
\end{align*}
which is equivalent to $c_1(Y) = 0$ as an integral cohomology class. 
\end{example}

\begin{example}
We follow a construction found in \cite{cr11} to give finite quotients of Calabi-Yau hypersurfaces in weighted projective spaces, which allows us to to realize the elliptic orbifolds (ii) through (iv) as finite quotients of cubic curves in $\mathbb{CP}^2$ (c.f. \cite{ks11}). A polynomial $F$ in $(n+1)$ variables is called quasi-homogeneous if there is a $n$-tuple of weights $(c_0, \ldots, c_n)$ such that for any scalar $\lambda$ we have 
\[
F(\lambda^{c_0}x_0, \ldots, \lambda^{c_n}x_n) = \lambda F(x_1, \ldots, x_n).
\]
We assume that $F$ is non-degenerate in the sense that $F$ defines an isolated singularity at the origin, and we also assume that $F$ is of Calabi-Yau type meaning $\sum_{i} c_i = 1$. Then the equation $F = 0$ defines a Calabi-Yau hypersurface $X_F$  in the weighted projective space $\mathbb{CP}[q_0, \ldots, q_n]$ where $q_i = c_i/d$ for some common denominator $d$.  (For example, if $F = x^2y + y^3 + xz^2$, then $F$ is quasihomogeneous with weights $(1/3, 1/3, 1/3)$, and hence $F$ defines a Calabi-Yau hypersurface in $\mathbb{CP}^2 = \mathbb{CP}[1,1,1]$.) Let $G_{\text{max}}$ denote the diagonal symmetry group 
\[
G_{\text{max}} = \{\text{Diag}(\lambda_0, \ldots, \lambda_N) : F(\lambda_0x_0, \ldots, \lambda_n x_0) = F(x_0, \ldots, x_n)\}.
\]
This group contains the element $J = \text{Diag}(e^{2\pi i c_0}, \ldots, e^{2\pi i c_n})$ which acts trivially on $X_F$. A subgroup $G$ satisfying $J \in G \subset G_{\text{max}}$ acts on $X_F$ with kernel $\langle J \rangle$. Hence the group $\tilde{G} = G/\langle J \rangle$ acts faithfully on $X_F$, and one obtains an orbifold global quotient $\mathcal{X} = [X_F/ \tilde{G}]$. This orbifold global quotient has $c_1(\mathcal{X}) = 0$ as a real cohomology class. In particular, the canonical sheaf $\omega_{\mathcal{X}}$ is not necessarily trivial, but some power of it is, so that some multiple of $c_1(\mathcal{X})$ vanishes as an integral cohomology class.

One can realize the elliptic orbifolds from Example \ref{ex:eo} as such quotients of cubic curves in $\mathbb{CP}^2$ (c.f. \cite{ks11}): 
\begin{enumerate}
\item[(ii)]  Consider the curve $X_F \subset \mathbb{CP}^2$ determined by the cubic polynomial $F = x^2y + y^3 + xz^2$. The group $\tilde{G} = \mathbb{Z}_4$ acts on $X_F$ via 
\[
\xi \cdot [x,y,z] = [\xi^2 x, y, \xi z]
\]
so that the quotient orbifold $[X_F/\tilde{G}]$ is the elliptic orbifold $\mathbb{P}_{4,4,2}$. Indeed the points whose stabilizers have orders $4, 4$ and $2$ are the images of the points  $[1,0,0]$, $[0,0,1]$, and $[1, \sqrt{-1}, 0]$. 
\item[(iii)] Consider the curve $X_F \subset \mathbb{CP}^2$ determined by the cubic polynomial $F = x^3 + y^3 + xz^2$. The group $\tilde{G} = \mathbb{Z}_6$ acts on $X_F$ via 
\[
\xi \cdot[x,y,z] = [\xi^4 x, y, \xi z]
\]
so that the quotient orbifold is the elliptic orbifold $\mathbb{P}_{6,3,2}$. Indeed, the points whose stabilizers have orders 6, 3, and 2 are the images of the points $[0,0,1], [1,0, \sqrt{-1}]$, and $[x, -1, 0]$ where $x$ is a third root of unity. 
\item[(iv)] Consider the curve $X_F \subset \mathbb{CP}^2$ determined by the cubic polynomial $F = x^3 + y^3 + z^3$. The group $\tilde{G} = \mathbb{Z}_3 \times \mathbb{Z}_3$ acts on $X_F$ via 
\[
(\xi_1, \xi_2) \cdot [x,y,z] = [x, \xi_1 y, \xi_2 z]
\] 
so that the quotient orbifold $[X_F/ \tilde{G}]$ is the elliptic orbifold $\mathbb{P}_{3,3,3}$. Indeed the points with stabilizers of size three are the images of the points $[-1, 0, z], [0,-1, z]$ and $[-1, y, 0]$ where $y,z$ denote third roots of unity. 
\end{enumerate}

In the general case, to each finite quotient Calabi-Yau orbifold $[X_F/\tilde{G}]$, there is an associated Berglund-Hubsch-Krawiz mirror $[X_F^T/\tilde{G}^T]$, which is another Calabi-Yau orbifold that is dual to the other one in the sense that there are symmetric isomorphisms at the level of certain cohomological groups, namely Chen-Ruan orbifold cohomology (see \cite{bh93} for the proposal of this ``classical mirror symmetry conjecture'' and \cite{cr11} for the proof.)
\end{example}

\begin{example}
One can also consider Calabi-Yau hypersurfaces in toric varieties defined by polyhedra, and in the case that the defining polyhedron is reflexive, there is a so-called Batyrev mirror, which is another Calabi-Yau hypersurface dual to the original hypersurface in the sense of mirror symmetry from mathematical physics \cite{b94}. 

More precisely, let $N$ be a lattice, and $M$ the corresponding dual lattice. A full-dimensional integral polytope $\Delta$ in $M$ is reflexive if 
\begin{enumerate}
\item[(i)] $0$ belongs to the interior of $\Delta$
\item[(ii)] there are vectors $v_F \in N$ associated to each codimension-1 face $F$ of $\Delta$ such that 
\[
\Delta = \{m \in M_{\mathbb{R}} : \langle m, v_F \rangle \geqslant 1 \; \text{for each $F$}\}.
\]
\end{enumerate}
Such a polytope determines a toric variety $X_\Delta$ which is Gorenstein and Fano. The anticanonical sheaf corresponds to the divisor $\sum_\rho D_\rho$ as $\rho$ ranges over the one-dimensional cones in the normal fan $\Sigma_\Delta$. By adjunction, the zero locus of a generic section of the anti-canonical sheaf determines a hypersurface $\mathcal{X}$ with trivial canonical sheaf so that $c_1(\mathcal{X})$ is zero as an integral cohomology class. Let $\mathcal{F}(\Delta)$ denote this family of Calabi-Yau hypersurfaces. 

This family of Calabi-Yau hypersurfaces $\mathcal{F}(\Delta)$ is dual to another family of Calabi-Yau hypersurfaces in the following sense. The polar dual of $\Delta$ is given by 
\[
\Delta^\circ = \{n \in N_{\mathbb{R}} : \langle m , n \rangle \geqslant -1 \; \text{for all $m \in \Delta$}\}
\]
and is reflexive if and only if $\Delta$ is. It follows that $\Delta^\circ$ also determines a Gorenstein toric Fano variety $X_{\Delta^\circ}$. In this way, one obtains a dual family $\mathcal{F}(\Delta^\circ)$ of Calabi-Yau hypersurfaces determined by generic sections of the anticanonical sheaf of $X_{\Delta^\circ}$. The involution taking $\mathcal{F}(\Delta)$ to $\mathcal{F}(\Delta^\circ)$ satisfies properties of the mirror duality in physics (see \cite{b94}). 

\end{example}

\begin{example}
One can generalize the previous example to complete intersection subvarieties of toric varieties and their mirrors \cite{bb94}. 

Let $X_\Delta$ be a Gorenstein toric Fano variety determined by a reflexive full dimensional integral polytope $\Delta$, and let $\Sigma_\Delta$ denote the corresponding normal fan. Let $E = \{e_1, \ldots, e_r\}$ denote the set of vertices of $\Delta$. A representation $E  = E_1 \cup \cdots \cup E_s$ as the disjoint union of subsets $E_1, \ldots, E_s$ is called a nef-partition if there are integral convex $\Sigma_\Delta$-piecewise linear functions $\varphi_1, \ldots, \varphi_s$ on $M_{\mathbb{R}}$ satisfying 
\[
\varphi_i(e_j) = \begin{cases}
1 & e_j \in E_i \\
0 & \text{otherwise}
\end{cases}.
\]
Such a partition induces a representation of the anticanonical divisor as the sum of $s$ Cartier divisors $\sum_{i=1}^s D_i$ which are nef. A  choice of an $s$-tuple of generic sections of the sheaves corresponding to these divisors gives rise to a complete intersection subvariety $\mathcal{X}$ of $X_\Delta$ which has trivial canonical sheaf and hence $c_1(\mathcal{X})$ is zero as an integral cohomology class. 

There is a duality on such complete intersections which can be described as follows. Let $E = E_1 \cup \cdots \cup E_s$ be a nef partition of the vertices of $\Delta$. For each $i$, let $\Delta_i'$ denote the convex polyhedron 
\[
\Delta_i' = \{n \in N_{\mathbb{R}} : \langle m, n \rangle \geqslant - \varphi_i(x) \}.
\]
Then it can be shown that the lattice polyhedron $\Delta'$ defined by  
\[
\Delta' = \text{Conv}(\Delta_1' \cup \cdots \cup \Delta_s').
\]
is reflexive \cite{b93}. Let $E'$ denote the collection of vertices of $\Delta'$. For each $i$, if $E'_i$ denotes the collection of vertices of $\Delta_i'$, then $E' = E_1' \cup \cdots \cup E_s'$ is a nef-partition of $E'$. In this way, the set of reflexive polyhedra together with nef-partitions enjoys a natural involution 
\[
(\Delta; E_1, \ldots, E_s) \mapsto (\Delta'; E_1', \ldots, E_s').
\]
It is shown in \cite{bb94} that this involution gives rise to mirror symmetric families of Calabi-Yau complete intersections in Gorenstein toric Fano varieties. 
\end{example}

\section*{Acknowledgements}

The author would like to thank Eleanora Di Nezza, Hans-Joachim Hein, Claude LeBrun, Daniel Litt, Chiu-Chu Melissa Liu, Jian Song, Gabor Szekelyhidi, and Valentino Tosatti for helpful discussions and suggestions. The author is also grateful to Chiu-Chu Melissa Liu and D.H. Phong for encouragement and support. The author is supported by NSF grant DGE 16-44869.

\bibliography{CalabiOrbifoldBib}

\begin{thebibliography}{10}

\bibitem{alr07}
A.~Adem, J.~Leida, and Y.~Ruan.
\newblock {\em Orbifolds and String Topology}, volume 171.
\newblock Cambridge University Press, Cambridge, 2007.

\bibitem{a78}
T.~Aubin.
\newblock \'equations du type {M}onge-{A}mp\'ere sur les variet\'es
  k\"ahleriennes compactes.
\newblock {\em Bull. Sci. Math.}, 102(1):63--95, 1978.

\bibitem{baily56}
W.~Baily.
\newblock The decomposition theorem for {V}-manifolds.
\newblock {\em American Journal of Mathematics}, 78(4):862--888, 1956.

\bibitem{baily57}
W.~Baily.
\newblock On the imbedding of {V}-manifolds in projective space.
\newblock {\em American Journal of Mathematics}, 79(2):403--430, 1957.

\bibitem{b94}
V.~Batryev.
\newblock Dual polyhedra and mirror symmetry for {C}alabi-{Y}au hypersurfaces
  in toric varieties.
\newblock In {\em J. Alg. Geom}, volume~3, pages 493--535, 1994.

\bibitem{bb94}
V.~Batyrev and L.~Borisov.
\newblock Dual cones and mirror symmetry for generalized {C}alabi-{Y}au
  manifolds.
\newblock {\em arXiv preprint alg-geom/9402002}, 1994.

\bibitem{bh93}
P.~Berglund and T.~H{\"u}bsch.
\newblock A generalized construction of mirror manifolds.
\newblock {\em Nuclear Physics B}, 393(1-2):377--391, 1993.

\bibitem{blocki12}
Z.~B{\l}ocki.
\newblock The {C}alabi-{Y}au theorem.
\newblock {\em Complex Monge-Amp{\`e}re Equations and Geodesics in the Space of
  K{\"a}hler Metrics}, pages 201--227, 2012.

\bibitem{b93}
L.~Borisov.
\newblock Towards the mirror symmetry for {C}alabi-{Y}au complete intersections
  in {G}orenstein toric {F}ano varieties.
\newblock {\em arXiv preprint alg-geom/9310001}, 1993.

\bibitem{bg08}
C.~Boyer and K.~Galicki.
\newblock {\em Sasakian geometry}.
\newblock Oxford Univ. Press, 2008.

\bibitem{calabi54}
E.~Calabi.
\newblock The space of k{\"a}hler metrics.
\newblock {\em Proc. Internat. Congress Math. Amsterdam}, 2:206--207, 1954.

\bibitem{campana04}
F.~Campana.
\newblock Orbifoldes \'a premi\'ere classe de {C}hern nulle.
\newblock {\em arXiv preprint math/0402243}, 2004.

\bibitem{chiang90}
Y.-J. Chiang.
\newblock Harmonic maps of {V}-manifolds.
\newblock {\em Annals of Global Analysis and Geometry}, 8(3):315--344, 1990.

\bibitem{cr11}
A.~Chiodo and Y.~Ruan.
\newblock Lg/cy correspondence: {T}he state space isomorphism.
\newblock {\em Advances in Mathematics}, 227(6):2157--2188, 2011.

\bibitem{cls11}
D.~A. Cox, J.~B. Little, and H.~K. Schenck.
\newblock {\em Toric varieties}.
\newblock American Mathematical Soc., 2011.

\bibitem{dk01}
Jean-Pierre Demailly and J{\'a}nos Koll{\'a}r.
\newblock Semi-continuity of complex singularity exponents and
  k{\"a}hler--einstein metrics on fano orbifolds.
\newblock In {\em Annales Scientifiques de l'{\'E}cole Normale Sup{\'e}rieure},
  volume~34, pages 525--556. Elsevier, 2001.

\bibitem{evans10}
L.~C. Evans.
\newblock {\em Partial Differential Equations}.
\newblock Graduate studies in mathematics. American Mathematical Society, 2010.

\bibitem{gt15}
D.~Gilbarg and N.~S. Trudinger.
\newblock {\em Elliptic partial differential equations of second order}.
\newblock Springer, 2015.

\bibitem{joyce00}
D.~Joyce.
\newblock {\em Compact manifolds with special holonomy}.
\newblock Oxford University Press on Demand, 2000.

\bibitem{kf70}
A.~N. Kolmogorov and S.~V. Fomin.
\newblock {\em Introductory Real Analysis}.
\newblock Dover, New York, 1970.

\bibitem{ks11}
M.~Krawitz and Y.~Shen.
\newblock {L}andau-{G}inzburg/{C}alabi-{Y}au correspondence of all genera for
  elliptic orbifold {$\mathbb{P}^1$}.
\newblock {\em arXiv preprint arXiv:1106.6270}, 2011.

\bibitem{narasimhan68}
R.~Narasimhan.
\newblock {\em Analysis on Real and Complex Manifolds}.
\newblock North-Holland, Amsterdam, 1968.

\bibitem{siu12}
Y.-T. Siu.
\newblock {\em Lectures on Hermitian-Einstein metrics for stable bundles and
  {K}{\"a}hler-{E}instein metrics: delivered at the German Mathematical Society
  Seminar in D{\"u}sseldorf in June, 1986}, volume~8.
\newblock Birkh{\"a}user, 2012.

\bibitem{szekelyhidi14}
G.~Sz{\'e}kelyhidi.
\newblock {\em An Introduction to Extremal Metrics}, volume 152.
\newblock Graduate Studies in Mathematics, Providence, 2014.

\bibitem{tian12}
G.~Tian.
\newblock {\em Canonical metrics in {K}{\"a}hler geometry}.
\newblock Birkh{\"a}user, 2012.

\bibitem{yau78}
S.-T. Yau.
\newblock On the {R}icci curvature of a compact {K}{\"a}hler manifold and the
  complex {M}onge-{A}mp{\'e}re equation, {I}.
\newblock {\em Communications on pure and applied mathematics}, 31(3):339--411,
  1978.

\end{thebibliography}
\bibliographystyle{plain}

\end{document}